\documentclass[11pt]{article}
\usepackage[latin1]{inputenc}
\usepackage{amsmath,amsthm}
\usepackage{amssymb}
\usepackage{graphicx}

\newcommand{\EE }{\mathsf{E}}
\newcommand{\PP}{\mathsf{P}}
\newcommand\egaldef{\stackrel{\mbox{\rm\tiny def}}{=}}
\newcommand\eqdist{\stackrel{\mbox{\rm\tiny d}}{=}}
\renewcommand{\epsilon}{\varepsilon}
\newcommand{\FCFS}{_{\mathrm{FCFS}}}
\newcommand{\ROS}{_{\mathrm{ROS}}}

\newcommand{\e}{\mathrm{e}}
\renewcommand{\d}{\mathrm{d}}
\newcommand\1{\leavevmode\hbox{\rm \small1\kern-0.35em\normalsize1}}
\newcommand\ind[1]{\1_{\{#1\}}}

\makeatletter
\@addtoreset{equation}{section}
\makeatother

\newtheorem{lemma}{Lemma}[section]
\newtheorem{theorem}[lemma]{Theorem}
\newtheorem{corollary}[lemma]{Corollary}
\newtheorem{prop}[lemma]{Proposition}
\theoremstyle{remark}
\newtheorem{remark}{Remark}[section]

\usepackage{latexsym}
\usepackage{a4wide}

\newcommand{\bp}{Z}
\newcommand{\st}{B}
\newcommand{\sfw}{B^{fw}}
\newcommand{\brp}{{\bp}^{rp}}
\newcommand{\srp}{B^{rp}}

\newcommand{\dd}{{\rm d}}
\newcommand{\beeq}{\begin{equation}}
\newcommand{\bear}{\begin{eqnarray}}
\newcommand{\bearno}{\begin{eqnarray*}}
\newcommand{\eneq}{\end{equation}}
\newcommand{\enar}{\end{eqnarray}}
\newcommand{\enarno}{\end{eqnarray*}}

\begin{document}

\title{Waiting time asymptotics in the single server queue with service in random order}
\author{O.J. Boxma
  \thanks{Department of Mathematics \& Computer Science and EURANDOM;
          Eindhoven University of Technology, 
          P.O. Box 513, 5600 MB Eindhoven, The Netherlands}
  \thanks{CWI, P.O. Box 94079, 1090 GB Amsterdam, The Netherlands}
\and S.G. Foss
  \thanks{Department of Actuarial Mathematics and Statistics,
          Heriot-Watt University,
          Riccarton, Edinburgh, EH14 4AS UK} 
\and J.-M. Lasgouttes
  \thanks{INRIA, Domaine de Voluceau, Rocquencourt, 
          BP 105, 78153 Le Chesnay Cedex, France} 
\and R. N\'u\~nez Queija\footnotemark[1] \footnotemark[2]}
\maketitle

\begin{abstract}
We consider the single server queue with service in random order. For
a large class of heavy-tailed service time distributions, we determine
the asymptotic behavior of the waiting time distribution. For the
special case of Poisson arrivals and regularly varying service time
distribution with index $-\nu$, it is shown that the waiting time
distribution is also regularly varying, with index $1-\nu$, and the
pre-factor is determined explicitly.

Another contribution of the paper is the heavy-traffic analysis of the
waiting time distribution in the $M/G/1$ case. We consider not only
the case of finite service time variance, but also the case of
regularly varying service time distribution with infinite variance. 


\bigskip
\noindent
\emph{Acknowledgement}: J.-M. Lasgouttes did most of his research for
the present study while spending a sabbatical at EURANDOM in
Eindhoven. O.J. Boxma and S.G. Foss gratefully acknowledge the support
of INTAS, project 265 on ``The mathematics of stochastic networks''.
\end{abstract}
\section{Introduction}

We consider a single server queue that operates under the Random Order
of Service discipline (ROS; also SIRO = Service In Random Order): At
the completion of a service, the server randomly takes one of the
waiting customers into service. Research on the ROS discipline has a
rich history, inspired by its natural occurrence in several problems
in telecommunications. The $M/M/1$ queue with ROS was studied by
Palm~\cite{Palm}, Vaulot~\cite{Vaulot} and
Pollaczek~\cite{Pollaczek46,Pollaczek59}. Burke~\cite{Burke} derived
the waiting time distribution in the $M/D/1$ case. An expression for
the (Laplace-Stieltjes transform of the) waiting time distribution for
the $M/G/1$ case was obtained by Kingman~\cite{Kin:2} 
and Le Gall~\cite{LeGall}; the former author also studied the
\emph{heavy-traffic} behavior of the waiting time distribution, when
the service times have a finite variance. Quite recently,
Flatto~\cite{Flatto} derived detailed tail asymptotics of the waiting
time in the $M/M/1$ case. As pointed out in Borst et al.~\cite{BBMN},
this immediately yields detailed tail asymptotics of the sojourn time
in the $M/M/1$ queue with Processor Sharing, because these two
quantities are closely related in a single server queue with
exponential service times.

In the present study, we are also interested in waiting time tail
asymptotics of single server queues with ROS\@. However, here we
concentrate on the case of \emph{heavy-tailed} service time
distributions. The motivation for this study is twofold. Firstly, an
abundance of measurement studies regarding traffic in communication
networks like local area networks and the Internet has made it clear
that such traffic often has heavy-tailed characteristics. It is
therefore important to investigate the impact of such traffic on
network performance and to determine whether possibly adverse effects
can be overcome by employing particular traffic management schemes.
One possibility is to modify the `service discipline' (i.e.,
scheduling mechanism); this may lead to a significant change in
performance~\cite{BBN}.

Secondly, in real life there are many situations in which service is
effectively given in random order. Our own interest in ROS was
recently revived in a joint project with Philips Research concerning
the performance analysis of cable access networks. Collision
resolution of user requests for access to the common transmission
channel is being handled by a Capetanakis-Tsybakov-Mikhailov type tree
protocol~\cite{BG}. That collision resolution protocol handles the
requests in an order that is quite close to ROS~\cite{BDR}.

We now present an outline of the organization and main results of the
paper. Section~\ref{prelim} contains preliminary results on the busy
period and waiting time tail behavior in the $GI/G/1$ queue with a
non-preemptive and non-idling service discipline. They are used in
Section~\ref{global} to study the waiting time tail for the $GI/G/1$
queue with service in random order. The tail of the service time
distribution is assumed to be in the class $\cal{L} \bigcap \cal{D}$.
This class contains the class of regularly varying distributions;
these two classes, and others, are briefly discussed in Appendix~A. We
sketch a probabilistic derivation of the asymptotic behavior of the
waiting time distribution, deferring a detailed derivation to
Appendix~\ref{app:main}. For large $x$, $\PP(W\ROS>x)$ is written as a
sum of four terms, each of which has a probabilistic interpretation.
These interpretations are based on the knowledge that, for sums of
independent random variables with a subexponential distribution, the
most likely way for the sum to be very large is that \emph{one of the
summands} is very large (similar ideas were developed in \cite{BF} for a class
of stochastic networks -- see the so-called `Typical Event Theorem' there).
For example, the first of the four terms
equals $\rho$ times the probability that a residual service time is
larger than $x$, $\rho$ denoting the traffic load. The probabilistic
interpretation is that one possibility for the waiting time of a
tagged customer to be larger than some large value $x$ is, that the
residual service time of the customer in service upon his arrival
exceeds $x$. The other three terms are more complicated, taking into
account possibilities like: A customer with a very large service time
has already left when the tagged customer arrived, but it has left a
very large number of customers behind --- and the tagged customer has
to wait for many of those (and newly arriving) customers.

In the subsequent sections we restrict ourselves to the
case of Poisson arrivals. In the case of an $M/G/1$ queue with
regularly varying service time distribution, we are able to obtain
detailed tail asymptotics for the waiting time distribution, in two
different ways: (i) in Section~\ref{rvlst} we apply a powerful lemma
of Bingham and Doney~\cite{BD} for Laplace-Stieltjes transforms (LST)
of regularly varying distributions to an expression of Le
Gall~\cite{LeGall} for the waiting time LST in the $M/G/1$ queue with
ROS\@; (ii) in Section~\ref{rvprob} we work out the general tail
asymptotics of Section~\ref{global} for this case. Either way, the waiting time
tail is proven to exhibit the following behavior in the regularly varying case:
\begin{equation}
\PP(W\ROS>x) \sim \frac{\rho}{1-\rho}h(\rho) \PP(B^{fw} >x), ~~~ x \rightarrow \infty .
\label{WROS}
\end{equation}
Here, and throughout the paper, $f(x) \sim g(x)$ denotes $\lim_{x
\rightarrow \infty} f(x)/g(x) = 1$; $h(\rho)$ is specified in
Formulas (\ref{311}) and (\ref{312}).
$B^{fw}$ denotes the forward recurrence time of the service times,
i.e., the residual service time. It is well-known that, with $B$
denoting an arbitrary service time,
\begin{equation}
\PP(B^{fw}>x) = \int_x^{\infty}  \frac{\PP(B>u)}{\EE  B} {\rm d}u , ~~~ x \geq 0.
\end{equation}
Note that, except for Poisson arrivals, $B^{fw}$ has a different
distribution than the residual service requirement of the customer in
service at arrival epochs.

Formula (\ref{WROS}) should be compared with the
waiting time tail asymptotics in the $M/G/1$ FCFS
case~\cite{Pakes75}:
\begin{equation}
\PP(W\FCFS>x) \sim \frac{\rho}{1-\rho} \PP(B^{fw} >x), ~~~ x \rightarrow \infty .
\end{equation}
We shall show that $h(\rho) \leq 1$,
which implies that ROS yields a (slightly) lighter tail than FCFS.

In Section~\ref{heavytraffic} we
allow the service time distribution to be completely general.
We study the waiting time distribution in the case
of heavy traffic (traffic load $\rho \uparrow 1$).
When the service time variance is finite, we retrieve
a result of Kingman~\cite{Kin:2}.
When the service time distribution is regularly varying with infinite variance,
we exploit a result of~\cite{BC} to derive a new heavy-traffic 
limit theorem.

The paper ends with four appendices. Appendix~\ref{app:classes} discusses several classes of heavy-tailed distributions.
Appendices~\ref{app:srp} and~\ref{app:main} provide the proofs of two theorems.
In Appendix~\ref{app:constarr} we state and prove a lemma that is not explicitly used in the paper.
However, it has been very useful in guiding us to the proofs of our main results.
Essentially, the lemma states that when interested in events 
involving a large service time, we may in fact ignore the 
randomness in the arrival process.

\begin{remark}
A different way of randomly choosing a customer for service is the following.
Put an arriving customer, who finds $n$ waiting customers,
with probability $\frac{1}{n+1}$ in one of the positions $1,2,\dots,n+1$,
and serve customers according to their order in the queue.
Fuhrmann and Iliadis~\cite{FI} prove that this discipline gives rise to
exactly the same waiting time distribution as ROS.
\end{remark}

\section{Preliminaries: Busy period and waiting time}
\label{prelim}

We first focus on the busy period of the $GI/G/1$ queue.
For the time being we may take the service discipline to be the familiar FCFS, since the busy period is the same for any 
non-idling discipline.
At the end of this section -- in Corollary~\ref{cor:non-idling} -- we use the results on the busy period to state a useful 
relation for the waiting time in any non-idling service discipline.

Let us introduce some notation.
The mean inter-arrival time is denoted with $\alpha$ and the random variable ${{\st}}$ stands for a generic service time, with 
mean $\EE  B=\beta$.
A generic busy period is denoted with the random variable ${\bp}$ and $\tau$ is the number of customers served in a busy 
period.
The residual busy period {\em as seen by an arriving customer} (i.e., the Palm version associated with arrivals) is denoted 
with $\brp$.
As before, we use $\sfw$ to denote a random variable with the forward recurrence time distribution of the service times.

For the $GI/G/1$ queue the proof of the following proposition is given in~\cite{fz}.
The definitions of the widely used classes ${\cal S}^*$, ${\cal IRV}$, ${\cal L}$ and ${\cal D}$ can be found in 
Appendix~\ref{app:classes}.
The first proposition can be specialized to the $M/G/1$ queue by substituting $\EE  \tau= \frac{1}{1-\rho}$.

\begin{prop}
\label{prop:bp}
If ${\st} \in {\cal S}^{*}$, then,
for any $0<c_1 < 1 < c_2$, 
\begin{equation} 
\lim \sup_{n\to\infty}
\frac{\PP ({\tau}>n)}{\EE  \tau \PP({\st}>c_1 n \alpha(1-\rho ))}
\leq 1,
\end{equation}
and
\begin{equation} 
\lim \inf_{n\to\infty}
\frac{\PP ({\tau}>n)}{\EE  \tau \PP({\st}>c_2 n \alpha(1-\rho ))}
\geq 1
.
\end{equation}
Similarly, 
for any $0<d_1 < 1 < d_2$, 
\begin{equation} 
\lim \sup_{x\to\infty}
\frac{\PP ({\bp}>x)}{\EE  \tau \PP({\st}>d_1x(1-\rho ))}
\leq 1,
\end{equation}
and
\begin{equation} 
\lim \inf_{x\to\infty}
\frac{\PP ({\bp}>x)}{\EE  \tau \PP({\st}>d_2x(1-\rho ))}
\geq 1
.
\end{equation}
In particular, if ${\st} \in {\cal IRV}$ then
\begin{equation} \label{tau}
\PP ({\tau}>n) \sim \EE  \tau \PP({\st}>n \alpha (1-\rho ))
,\quad \mbox{as} \quad n\to\infty ,
\end{equation}
and
\begin{equation} \label{B}
\PP ({\bp}>x) \sim \EE  \tau \PP({\st}>x(1-\rho ))
,\quad \mbox{as} \quad x\to\infty .
\end{equation}
\end{prop}

The next proposition gives the asymptotics of the distribution of the {\em residual} busy period.
Heuristically speaking, it indicates that a large residual busy period requires exactly one large service requirement (in the past).
When analyzing waiting times (and residual service requirements) this result proves to be very useful as we shall see later.
In fact, we shall sharpen the statement of the proposition (in line with the heuristics) in Corollary~\ref{cor:bpres}.

\begin{prop}
\label{prop:bpres}
If ${\st}\in {\cal L} \bigcap {\cal D}$, then
\begin{eqnarray} \label{probab}
\PP (\brp>x) &\sim & 
\sum_{m=1}^{\infty}
\PP ({\st}_{-m} > (x+m \alpha) (1-\rho ))
\\
\label{Bres}
&\sim& \frac{\rho}{1-\rho} \PP({\st}^{fw}>x(1-\rho )),
\qquad
x\rightarrow\infty,
\end{eqnarray}
where ${\st}_{-m}$ is the service time of the $m$-th customer (counting backwards)   
in the elapsed busy period.
\end{prop}

\begin{proof}
Let us concentrate on the residual busy period as seen by an arbitrary customer (``customer~0'') arriving at time~$T_0=0$.
With $V_{-m}$ we denote the amount of work in the system found by customer $-m$ and by $\bp_{-m}$ the consecutive busy 
period if $V_{-m}=0$.
Furthermore, for $m>0$, $T_{-m}$ is the time between the arrival of customer~$-m$ and time~0 and $T_{m}$ is the time of 
arrival of the $m$-th customer after time~0.
We may write
\bearno
    \PP\left(\bp^{rp}>x\right)
    &=&
    \sum_{m=1}^\infty\PP\left(V_{-m}=0,\bp_{-m}>T_{-m}+x\right)
    \\
    &=&
    \sum_{m=1}^\infty\PP\left(V_{-m}=0\right)\PP\left(\bp_{-m}>T_{-m}+x\right)
    \\
    &=&
    \frac{1}{\EE \tau} \sum_{m=1}^\infty \PP\left(\bp_{0}>T_{m}+x\right).
\enarno
From this, the proof is quite straightforward in the case of
constant inter-arrival times $T_{m} \equiv m \alpha$. 
In that case it follows from Proposition~\ref{prop:bp}, that for any $\delta\in(0,1)$ and $d_1>1$ there is an $x_0$ such 
that 
$$
\PP (\bp_{0}> T_{m} +x) \equiv
\PP (\bp > m\alpha +x) \leq 
(1+\delta)\EE  \tau \PP\left( \st> d_1(x+m\alpha)(1-\rho )\right),
$$
for all $x>x_0$ and $m\ge1$.
For $x>x_0$ this gives
\bearno
    \PP\left(\bp^{rp}>x\right)
    &\leq&
    (1+\delta)\sum_{m=1}^\infty \PP\left( \st> d_1(x+m\alpha)(1-\rho )\right)
    \\
    &\sim&
    \frac{(1+\delta)\rho}{d_1(1-\rho)} \PP\left(\sfw>x(1-\rho)\right).
\enarno
Now let $\delta\rightarrow0$, $d_1\rightarrow1$ and use that $\sfw\in {\cal IRV}$ (by Property~(\ref{7.5}) in Appendix~\ref{app:classes}) to obtain the desired upper bound
\bearno
    \PP\left(\bp^{rp}>x\right)
    &\leq&
    (1+o(1))\frac{\rho}{1-\rho} \PP\left(\sfw>x(1-\rho)\right),
    \qquad x\rightarrow\infty.
\enarno
The lower bound can be derived similarly.

When inter-arrival times are not constant the proof is more involved since $Z_{0}$ and $T_{m}$ are not independent.
First we note that since $B\in {\cal L}$, then, for any $\varepsilon > 0$, 
\begin{equation}
\label{eq:noexpdecay}
 \e^{-\varepsilon x} = o\left(\PP\left(B>x\right)\right),
 \qquad x\to\infty.
\end{equation}
We shall now develop upper and lower bounds for $\sum_{m=0}^\infty \PP\left(Z_0>T_{m}+x\right)$, which coincide for 
$x\rightarrow\infty$. 

{\em Upper Bound}. 
For any $\varepsilon \in (0,1)$, 
\[
 \PP (Z_0>T_m+x) \leq \PP \left(Z_0> -\varepsilon x + m\alpha (1-\varepsilon  )+x\right) 
 + \PP (T_m \leq -\varepsilon x+m\alpha (1-\varepsilon )). 
\]
From Proposition~\ref{prop:bp}, for any $d_1\in (0,1)$, 
\begin{eqnarray*} 
 \sum_{m=0}^\infty \PP (Z_0> x(1-\varepsilon ) + m\alpha (1-\varepsilon )) 
 &\leq & (1+o(1)) \EE  \tau \sum_{m=0}^\infty \PP 
 (B> d_1(1-\varepsilon ) (1-\rho ) (x+m\alpha)),
\end{eqnarray*} 
as $x\to\infty$.
For notational convenience we set $c_1 = d_1(1-\varepsilon )$ and note that $c_1\uparrow 1$ when 
$d_1\uparrow 1$ and $\varepsilon \downarrow 0$. 
Furthermore,
\begin{eqnarray*} 
 \sum_{m=0}^\infty \PP (B> c_1 (1-\rho ) (x+m\alpha))
 &\sim & 
 \frac{\rho}{c_1(1-\rho )} \PP 
 (B^{fw} > c_1(1-\rho )x). 
\end{eqnarray*} 
We use $t_n$ to denote the inter-arrival time of customer $n$ and
customer $n+1$, thus, $T_m=t_1+\cdots+t_m$. By the Chernoff inequality
we have, for any $r>0$,
\[
 \PP (T_m \leq -\varepsilon x + m\alpha(1-\varepsilon )) 
 =
 \PP \left( 
 \e^{-rT_m} \geq \e^{r\varepsilon x - rm\alpha(1-\varepsilon )}\right)
 \leq 
 \left( \EE  \e^{-rt_1} \right)^m \e^{rm\alpha(1-\varepsilon )-r\varepsilon 
   x}. 
\]
Since $\EE  t_1 = \alpha$ and $\varepsilon >0$, we can choose $r>0$ sufficiently small, such that 
\begin{equation}
\label{addi} 
 \e^{r\alpha (1-\varepsilon )}\EE \e^{-rt_1} < 1. 
\end{equation}
Then 
\begin{equation}
\label{eq:chernoff1}
 \sum_{m=0}^\infty \PP (T_m \leq -\varepsilon x + m\alpha(1-\varepsilon ))
 \leq \e^{-r\varepsilon x} \sum_{m=0}^\infty \left( \e^{r\alpha(1-\varepsilon )}\EE \e^{-rt_1}  \right) ^m =
 \frac{\e^{-r\varepsilon x} }{1-\e^{r\alpha(1-\varepsilon )}\EE \e^{-rt_1} } . 
\end{equation}
Thus, from~(\ref{eq:noexpdecay}),
$$ 
 \lim \sup_{x\to\infty} 
 \frac{\sum_{m=0}^\infty \PP (Z_0>T_m+x)}{ 
 \PP (B^{fw}>c_1 (1-\rho ) x)} 
 \leq \EE \tau \frac{\rho}{c_1(1-\rho )}. 
$$ 
Since $B^{fw}\in {\cal IRV}$ (Property~(\ref{7.5}) in Appendix~\ref{app:classes}), letting $c_1$ to $1$, we get 
 $$ 
 \lim \sup_{x\to\infty} 
 \frac{\sum_{m=0}^\infty \PP (Z_0>T_m+x)}{ 
 \PP (B^{fw}>(1-\rho ) x)} 
 \leq \EE  \tau \frac{\rho}{1-\rho },
$$ 
which concludes the upper bound.

{\em Lower Bound}. 
For any $\varepsilon \in (0,1)$, put 
$$ 
 n_{x,m} = \left\lfloor \frac{x(1+\varepsilon )}{\alpha} + \varepsilon m \right\rfloor,
$$ 
where $\lfloor y \rfloor$ denotes the integer part of a positive real number $y$.
Obviously,
\begin{eqnarray*} 
 \PP (Z_0>T_m+x) 
 &\geq & 
 \PP (Z_0>T_{m+n_{x,m}}, T_{m+n_{x,m}}-T_m \geq x)\\ 
 &\geq & 
 \PP (Z_0>T_{m+n_{x,m}}) - \PP (T_{m+n_{x,m}}-T_m <x)\\ 
 &=& 
 \PP (\tau > m+n_{x,m}) - \PP (T_{n_{x,m}}<x). 
 \end{eqnarray*} 
From Proposition~\ref{prop:bp}, for any $c_2>1$, 
$$ 
 \PP (\tau > m+n_{x,m}) \geq
 (1+o(1)) 
 \EE  \tau \PP (B> \alpha(1-\rho ) (m+n_{x,m})c_2). 
$$ 
Similar to (\ref{addi}), we can choose $r>0$ such that 
\[
 \e^{r\alpha (1+\frac{1}{2}\varepsilon )/(1+\varepsilon)}\EE \e^{-rt_1} < 1,
\]
so that
\begin{eqnarray*}
 \sum_{m=0}^\infty \PP (T_{n_{x,m}}<x) 
 &=& \sum_{m=0}^\infty \PP (\e^{-rT_{n_{x,m}}}>\e^{-rx}) 
 \leq
 \sum_{m=0}^\infty \e^{rx}\EE \e^{-rT_{n_{x,m}}}
 = \sum_{m=0}^\infty \e^{rx} \left(\EE \e^{-rt_1}\right)^{n_{x,m}}
 \\
 &\leq& \frac {\e^{r x} 
\left(\EE\e^{-rt_1}\right)^{x(1+\varepsilon)/\alpha  -1}}
 { 1-\left(\EE\e^{-r t_1}\right)^\varepsilon}
 \leq \frac {\e^{-\frac{1}{2}\varepsilon r x} } { 
\left(
1-\left(\EE\e^{-r t_1}\right)^\varepsilon
\right) \EE e^{-rt_1}}.
\end{eqnarray*} 
Thus, by~(\ref{eq:noexpdecay}), 
$$ 
 \sum_{m=0}^\infty \PP (Z_0>T_m+x) 
 \geq 
 (1+o(1)) 
 \frac{\rho \EE  \tau}{c_2(1+\varepsilon )(1-\rho )} 
 \PP (B^{fw} > (1-\rho ) x(1+\varepsilon )c_2). 
$$ 
Letting $c_2\downarrow 1$ and $\varepsilon \downarrow 0$, 
we get the desired result.

\end{proof}

\begin{remark}
For the class of RV tails the equivalence (\ref{B}) was proved by
Zwart~\cite{zwart}.  The asymptotic result (\ref{B}) also holds in a
class of so-called square-root insensitive subexponential
distributions under the additional condition that the second moment
of the inter-arrival time distribution is finite~\cite{jm}.  More
precisely, Jelenkovic et al. \cite{jm} established the following
result for the stable $G1/G1/1$ queue. If the following three
conditions are satisfied:
\begin{enumerate}
\item[(a)] The distribution of service times is square-root insensitive:
  $$
  {\bf P} (B>x+\sqrt{x}) \sim {\bf P} (B>x), \quad x\to\infty ;
  $$
\item[(b)] also, the distribution of $B$ belongs to the class of
  so-called {\it strong concave} ({\cal SC}) distributions -- which
  is a sub-class of ${\cal S}^*$;
\item[(c)] the distribution of inter-arrival times has a finite second
  moment;
\end{enumerate}
then (\ref{B}) holds. Under the same conditions, it may be shown
that the distribution tail of the number of customers served in a
busy period, $\tau$, has similar asymptotics:
$$
{\bf P} (\tau >n) \sim {\bf E} \tau {\bf P} (B > n\alpha (1-\rho)).
$$
Therefore, one can conclude that the asymptotics (2.7)-(2.8) are
also valid under conditions (a)-(c) above.  It would be worthwhile
to formulate and prove Corollaries~\ref{cor:bpres} and
\ref{cor:non-idling} and Theorems~\ref{thm:srp} and \ref{thm:main}
(the main theorem) for the class of service time distributions that
satisfy (a) and (b) above, under the restriction that the arrival
times satisfy (c).
\end{remark}

\medskip

Proposition~\ref{prop:bpres} states that, for large $x$, the events
$\bigcup_{m=1}^{\infty} \{ {\st}_{-m} > (x+m\alpha) (1-\rho )\}$ and
$\left\{\brp>x\right\}$ are equally likely. In the sequel we shall
need that these two events actually occur simultaneously (for large
$x$). This statement is made precise in the following corollary.

\begin{corollary}
\label{cor:bpres}
If ${\st}\in {\cal L} \bigcap {\cal D}$, then
\begin{eqnarray} 
\label{eq:bpsum}
\PP (\brp>x) &\sim & 
\sum_{m=1}^{\infty}
\PP (\brp>x, {\st}_{-m} > (x+m\alpha) (1-\rho ))
\\
&\sim & 
\label{eq:bpunion}
\PP (\brp>x, \bigcup_{m=1}^{\infty}\left\{ {\st}_{-m} > (x+m\alpha) (1-\rho )\right\})
,
\qquad
x\rightarrow\infty.
\end{eqnarray}
\end{corollary}

\begin{proof}
First we show that~(\ref{eq:bpsum}) implies~(\ref{eq:bpunion}).
Note that
\bearno
    &&\PP (\brp>x, \bigcup_{m=1}^{\infty}\left\{ {\st}_{-m} > (x+m\alpha) (1-\rho )\right\})\\
    &&\leq
    \sum_{m=1}^{\infty}
    \PP (\brp>x, {\st}_{-m} > (x+m\alpha) (1-\rho ))
    \\
    && \leq
    \PP (\brp>x, \bigcup_{m=1}^{\infty}\left\{ {\st}_{-m} > (x+m\alpha) (1-\rho )\right\})\\
    &&+   \sum_{m_1\neq m_2} 
\PP ({\st}_{-m_1} > (x+m_1\alpha) (1-\rho ),
    {\st}_{-m_2} > (x+m_2\alpha) (1-\rho )),
\enarno
where the last sum is not bigger than
\bearno
    &&
    \sum_{m_1=1}^\infty\sum_{m_2=1}^\infty\PP ( {\st}_{-m_1} > 
(x+m_1\alpha )(1-\rho )) \PP (
    {\st}_{-m_2} > (x+m_2\alpha) (1-\rho ))
    \\
      &  \sim &
    \left(\rho \PP(\sfw>x)\right)^2
    = o(\PP(\sfw>x)).
\enarno
Using Proposition~\ref{prop:bpres} this proves 
that~(\ref{eq:bpsum}) implies~(\ref{eq:bpunion}).

It remains to show that the right-hand side 
of~(\ref{eq:bpsum}) matches~(\ref{Bres}).
As before, we use $t_n$ to denote the inter-arrival time of customer $n$ and customer $n+1$.
Assume that for some constants $\epsilon>0$ and $R>0$ the following events occur
\begin{enumerate}
\item
$E^{\epsilon,R}_{m,1}(x):=\left\{\st_{-m} > (x+R)\frac{1-\rho+\epsilon(1+\rho)}{1-\epsilon}+m\alpha(1-\rho) +\epsilon m\alpha(1+\rho) 
+ (1+\epsilon)\alpha + 2R\right\}$;
\item
$E^{\epsilon,R}_{m,2}:=\left\{\mbox{for all } n\geq1: \sum_{i=1}^n \st_{-m+i}\geq n\beta(1-\epsilon)- R\right\}$;
\item
$E^{\epsilon,R}_{m,3}:=\left\{\mbox{for all } n\geq1: \sum_{i=1}^n t_{-m+i}\leq n\alpha(1+\epsilon)+ R\right\}$;
\item
$E^{\epsilon,R}_{4}:=\left\{\mbox{for all } n\geq1: \sum_{i=1}^n t_{i}\geq n\alpha(1-\epsilon)- R\right\}$;
\end{enumerate}
then $V_n$ -- the amount of work seen upon arrival by customer $n$ -- satisfies, for $n>-m$,
\[
    V_n\geq \sum_{i=-m}^{n-1}(B_i-t_{i+1}) \geq (x+R)\frac{1-\rho+\epsilon(1+\rho)}{1-\epsilon}
-(n-1)\alpha(1-\rho+\epsilon(1+\rho)).
\]
Therefore all customers $n$ with
\[
    n-1<\frac{x+R}{ \alpha(1-\epsilon)},
\]
are in the same busy period, so that
\[
    \brp>\sum_{i=1}^nt_i\geq x.
\]
We thus have
\bearno
    \sum_{m=1}^\infty\PP\left(E^{\epsilon,R}_{m,1}(x)\cap E^{\epsilon,R}_{m,2}\cap E^{\epsilon,R}_{m,3}\cap 
E^{\epsilon,R}_{4}\right)
    &\leq&
    \sum_{m=1}^\infty\PP\left(\brp>x,B>(x+m\alpha)(1-\rho)\right)
    \\
    &\leq&
    \sum_{m=1}^\infty\PP\left(B>(x+m\alpha)(1-\rho)\right),
\enarno
and we are done if a lower bound for 
$\sum_{m=1}^\infty\PP\left(E^{\epsilon,R}_{m,1}(x)\cap E^{\epsilon,R}_{m,2}\cap E^{\epsilon,R}_{m,3}\cap 
E^{\epsilon,R}_{4}\right)$ 
is shown to be arbitrarily close to $(1+o(1))\frac{\rho}{1-\rho}\PP\left(\sfw>x(1-\rho)\right)$, as $x\rightarrow\infty$.
For any fixed $\delta>0$ we can find (by the strong law of large numbers) $\epsilon$ and $R$ such that
$\PP\left(E^{\epsilon,R}_{m,2}\right)\geq1-\delta$ and 
$\PP\left(E^{\epsilon,R}_{m,3}\cap E^{\epsilon,R}_{4}\right)\geq1-\delta$.
Thus, as $x\rightarrow\infty$,
\bearno
    &&\sum_{m=1}^\infty\PP\left(E^{\epsilon,R}_{m,1}(x)\cap E^{\epsilon,R}_{m,2}\cap E^{\epsilon,R}_{m,3}\cap        
    E^{\epsilon,R}_{4}\right)
    \,=\,\sum_{m=1}^\infty\PP\left(E^{\epsilon,R}_{m,1}(x)\right)\PP\left(E^{\epsilon,R}_{m,2}
    \right)\PP\left( E^{\epsilon,R}_{m,3}\cap     E^{\epsilon,R}_{4}\right)
    \\
    &&\geq
    (1-\delta)^2
    \sum_{m=1}^\infty\PP\left(E^{\epsilon,R}_{m,1}(x)\right)
    \\
    &&\sim\frac{(1-\delta)^2\rho}{1-\rho +\epsilon (1+\rho)}
    \PP\left(\sfw>(x+R)\frac{1-\rho+\epsilon(1+\rho)}{1-\epsilon} + (1+\epsilon)\alpha + 2R\right)
    \\
    &&\sim
    \frac{(1-\delta)^2\rho}{1-\rho +\epsilon (1+\rho)}
    \PP\left(\sfw>x\frac{1-\rho+\epsilon(1+\rho)}{1-\epsilon}\right),
\enarno
where we have used $\sfw\in{\cal L}$.
Now, first letting $\epsilon\rightarrow0$, using $\sfw\in{\cal IRV}$
(Property~(\ref{7.5}) in Appendix~\ref{app:classes}), and then $\delta\rightarrow0$ the
proof is completed.
\end{proof}

\begin{remark}
Expression~(\ref{eq:bpsum}) shows that the occurrence of a large residual busy period is due to a single large service time 
{\em in the past}.
This can be explained as follows.
The busy period is the sum of services of the customers in that busy period.
The number of customers in the busy period after time 0 (the point of arrival) is almost surely finite.
There are, however, infinitely many service times in the past, each of them being potentially large.
This leads to the integrated tail of the service time distribution.

\end{remark}

Besides the busy period and the residual busy period, there is a third entity whose distribution is the same for all non-idling and non-preemptive service disciplines:~$\srp$, the residual
service requirement of the customer in service (if any) upon arrival
of a new customer.
The tail asymptotics for the distribution of~$\srp$ are determined in the following theorem. 
Not only is the theorem of interest in itself, but several steps in its proof will also be
useful in proving our main result in Theorem~\ref{thm:main}.

\begin{theorem}
\label{thm:srp}
If $\st\in{\cal L}\bigcap{\cal D}$ then, for any non-preemptive and
non-idling service discipline,
\[
    \PP\left(\srp>x\right)
    \sim
    \rho \PP\left(\sfw>x\right) ,
\]
as $x\to\infty$.
\end{theorem}

\begin{proof}
See Appendix~\ref{app:srp}.
\end{proof}

\begin{corollary}
\label{cor:non-idling}
If ${\st}\in {\cal L} \bigcap {\cal D}$, then for any non-preemptive and non-idling service discipline, the waiting time 
$W$ and residual busy period $\bp^{rp}$ seen by a customer arriving to a stationary $GI/G/1$ queue satisfy $W\leq \brp$ 
a.s.~and therefore, as $x\to\infty$,
\begin{equation} \label{any}
\PP (W>x)
\sim
\sum_{m=1}^{\infty}
\PP \left(
W>x, {\st}_{-m} > (
x+m\alpha)
(1-\rho ) \right),
\qquad
x\rightarrow\infty.
\end{equation}
\end{corollary}

\begin{proof}
Since the service discipline is non-preemptive we have~$W\geq \srp$ almost surely, so that by Theorem~\ref{thm:srp},
\[
    \PP\left(W>x\right)
    \geq
    \PP\left(\srp>x\right)
    \sim
    \rho \PP\left(\sfw>x\right),
    \qquad x\to\infty.
\]
Thus, in the following, we may neglect terms which are $o\left(\PP\left(\sfw>x\right)\right)$.
Using $W\leq \brp$ (almost surely) and Corollary~\ref{cor:bpres} we therefore have (similar to the proof of Theorem~\ref{thm:srp})
\bearno
    \PP(W>x)
    &=& \PP\left(W>x, \brp>x\right)
    \\
    &\sim & 
    \PP (W>x, \brp>x, \bigcup_{m=1}^{\infty}\left\{ {\st}_{-m} > (x+m\alpha) (1-\rho )\right\})
    \\
    &\sim & 
    \sum_{m=1}^{\infty}
    \PP (W>x,\brp>x,{\st}_{-m} > (x+m\alpha) (1-\rho ))
    \\
    &=&
    \sum_{m=1}^{\infty}
    \PP (W>x,{\st}_{-m} > (x+m\alpha) (1-\rho ))
.
\enarno
\end{proof}

\section{Random Order of Service}
\label{global}

We now turn to the $GI/GI/1$ queue with Random Order of Service. We
start with analyzing the waiting time \emph{conditional} on the
initial queue length $q$, none of these customers having received
previous service. It is convenient to associate service times with
customers in their order of service instead of their order of arrival.
The customer which is served first has a service time $\st_1$, the
second has $\st_2$, etc. Denote with $W\ROS(q)$ the conditional waiting
time of an arbitrary customer in the queue.

The following lemma does not require any assumptions on the
distributions of service times and inter-arrival times. It shows that
$W\ROS(q)/q$ converges in distribution to a random variable whose
distribution has support $[0,\frac{\beta}{1-\rho}]$. Note that this
contrasts with the FCFS queue, in which case the corresponding quantity $W(q)/q$ (for the last
customer in line) converges to the constant $\beta$ almost surely.

\begin{lemma}
\label{lem:wq}
As $q\to\infty$,
\begin{equation} \label{as1}
\PP
\left(
W\ROS(q) > \frac{\beta q}{1-\rho}y \right)
\rightarrow
\left(
(1-y)^+ \right)^{\frac{1}{1-\rho}},
\end{equation}
uniformly in $y\in [0,\infty )$.
\end{lemma}

\begin{proof}
Since the limiting distribution is continuous and non-defective, by the monotonicity of probability distribution functions, it is sufficient to prove point-wise convergence.

Let us thus fix $y\in(0,1)$.
For $n=1,2,\ldots$, denote by $Q_n$ the number of customers in the
queue at the time instant of the $n$th service completion.
We define the event $A_1$ by
\[
A_1 :=
\left\{
Q_i \in [q(1-\varepsilon )
- i(1-\rho + \varepsilon ),
q(1+\varepsilon ) - i(1-\rho - \varepsilon )]
\quad \forall ~ i=1,2,\ldots ,\frac{q}{1-\rho }
\right\}.
\]
By the strong law of large numbers, for any $\varepsilon >0$, there exists $\tilde{q} 
\equiv \tilde{q}(\varepsilon )$ such that
$
\PP(A_1) \geq 
1-\varepsilon
$
for all $q\geq \tilde{q}$. 

Let $q_i^- =  q(1-\varepsilon )
- i(1-\rho + \varepsilon )$ and $q_i^+
= q(1+\varepsilon ) - i(1-\rho - \varepsilon )$ and denote $N(v) = \min \{ n ~:~ \sum_1^n
\st_i >v \}$; customer $N(v)$ is in service at time $v$. 
Defining
\[
A_2 :=
\left\{
N(v) \in \left[ \frac{v}{\beta}(1-\varepsilon )-R,
\frac{v}{\beta}(1+\varepsilon )+R \right] ,\quad
\forall ~ v\in \left[0, \frac{\beta q}{1-\rho}y\right] \right\} ,
\]
for any $\varepsilon >0$, we may choose
$R>0$ such that
$
\PP(A_2)
\geq 1-\varepsilon 
$.
Thus, $\PP (A_1\bigcap A_2) \geq 1-2\varepsilon $ and
$$
\PP \left(
W\ROS(q)>\frac{\beta q}{1-\rho}y
\right)= 
P(y)+ O(\varepsilon ),
$$
where
$$
P(y):=\PP
\left(
W\ROS(q)>\frac{\beta q}{1-\rho}y, A_1\bigcap A_2
\right) .
$$
We further define
$u=\frac{\beta q}{1-\rho}y$,
$n^-(u) = \frac{u}{\beta }(1-\varepsilon )-R$
and 
$n^+(u) = \frac{u}{\beta }(1+\varepsilon )+R$.
Since $\{ W\ROS(q) > \frac{\beta q}{1-\rho}y \}$ implies that customer $0$ was
not selected in the first $N\left( \frac{\beta q}{1-\rho}y\right)$ trials,
we have, as $q$ and $u$ tend to infinity keeping $y$ constant,
\begin{eqnarray*}
P(y) & \leq & \Pi_{i=1}^{n^-(u)}
\left(
1-\frac{1}{q_i^+}\right) \\
&=&
(1+o(1)) \exp \left(
- \sum_{i=1}^{n^-(u)} \frac{1}{q^+_i}\right) \\
&=&
(1+o(1))
\exp
\left(-
\int_0^{n^-(u)} \frac{1}{q(1+\varepsilon ) - v(1-\rho -\varepsilon )}
dv \right) \\
&=&
(1+o(1))
\left(
1-y\frac{(1-\varepsilon )(1-\rho -\varepsilon )}{
(1+\varepsilon )(1-\rho )}
\right)^{\frac{1}{1-\rho -\varepsilon }}\\
&=& (1+o(1))
(1-y+O(\varepsilon ))^{\frac{1}{1-\rho -\varepsilon}}.
\end{eqnarray*}
Similarly,
$$
P(y) \geq \Pi_{i=1}^{n^+(u)} \left(
1-\frac{1}{q_i^-}\right) - O(\varepsilon ) =
(1+o(1)) \left(
1-y-O(\varepsilon )
\right)^{\frac{1}{1-\rho-\varepsilon }}
- O(\varepsilon ).
$$
Letting $\varepsilon$ pass to $0$, we obtain~(\ref{as1}).
\end{proof}

The main result of our paper is stated in the next theorem.

\begin{theorem}
\label{thm:main}
In the GI/G/1 ROS queue with $\st\in{\cal L}\bigcap{\cal D}$, we have
\begin{eqnarray*}
\PP(W\ROS>x) &\sim &
\rho \PP(\st^{fw}>x)\\
&+&
\int_{0}^{cx}dv
\int_{(v\alpha+x)(1-\rho )}^{v\alpha+x} d\PP(B\leq z)
\left(
1-\frac{(x+v\alpha-z)(1-\rho )}{\rho z} \right)^{\frac{1}{1-\rho}} \\
&+&
\int_{cx}^{\infty}dv
\int_{v\alpha}^{v\alpha+x} d\PP(B\leq z)
\left(
1-\frac{(x+v\alpha-z)(1-\rho )}{\rho z} \right)^{\frac{1}{1-\rho}} 
\\
&+&
\int_{cx}^{\infty}dv
\int_{(v\alpha+x)(1-\rho )}^{v\alpha} d\PP(B\leq z)
\left(
1-\frac{x(1-\rho )}{z- v \alpha (1-\rho )} \right)^{\frac{1}{1-\rho}}
,
\end{eqnarray*}
where $c=\frac{1-\rho}{\alpha \rho}$.
\end{theorem}

\begin{proof}
In Appendix~\ref{app:main}.
\end{proof}

Letting $W\ROS^*$ be a random variable independent of $B$ with the limiting distribution of $W\ROS(q)/q$, as $q\to\infty$, we may conveniently rewrite the above as:
\bearno
\PP(W\ROS>x) &\sim &
\PP(\srp>x)\\
&+&
\int_{0}^{cx}dv
\PP\left((v\alpha+x)(1-\rho )< B\leq v \alpha +x,
W\ROS^*>\frac{\alpha (x+v\alpha-\st)}{\st} \right) \\
&+&
\int_{cx}^{\infty}dv
\PP\left(v\alpha< B\leq v\alpha+x,
W\ROS^*>\frac{\alpha (x+v\alpha-\st)}{\st} \right) 
\\
&+&
\int_{cx}^{\infty}dv
\PP\left((v\alpha+x)(1-\rho )<B\leq v\alpha,
W\ROS^*>\frac{\beta x}{\st- v\alpha(1-\rho )} \right)
.
\enarno
This allows for the following interpretation:
The waiting time is larger than $x$ when one of the following occurs:
\begin{enumerate}
\item
(first term)
The customer in service has a residual service time larger than $x$.
Recall that, by Theorem~\ref{thm:srp},~$\PP\left(B^{rp}>x\right)\sim\rho\PP\left(B^{fw}>x\right)$. 
\item
(second term)
A customer (with index~$-v$) that arrived at some time $-t=-\alpha v$ between time~$-\alpha cx$ and time~$0$ required a service $z$ 
larger than $(t+x)(1-\rho)$ but smaller than $t+x$.
The service times of other customers in the system at time~$-t$ are negligible compared to~$z$.
The large service time ends at time~$-t+z\in(0,x)$, leaving approximately~$z/\alpha$ competing customers in the system.
Customer~0 thus waits for approximately~$-t+z+W\ROS^*z/\alpha$.
Thus~$W\ROS^*$ needs to be larger than~$(x+t-z)\alpha/z$.
\item
(third term)
This term is similar to the previous.
Now, the large customer arrived at time~$-t<-\alpha cx$ with a service requirement~$z\in(t,t+x)$.
That customer thus leaves at time~$-t+z\in(0,x)$ with approximately~$z/\alpha$ customers in the system.
\item
(fourth term)
Again, the large customer arrived at time~$-t<-\alpha 
cx$, but leaves before time~0: its service requirement 
is~$z\in((t+x)(1-\rho),t)$.
Neglecting the size of the customer in service at time~0, the ``service lottery'' starts immediately upon 
arrival of customer~0.
The number of competing customers at time~0 is approximately~$t/\alpha-(t-z)/\beta$ which is the number of arrivals minus the 
number of departures between times~$-t$ and~0.
Therefore, the waiting time of customer~0 is larger than $x$ if~$W\ROS^*$ is larger 
than~$x/(t/\alpha-(t-z)/\beta)=\beta x/(z-t(1-\rho))$.

\end{enumerate}

\section{The $M/G/1$ queue with regularly varying service time distribution}
\label{rvlst}
In this section we restrict ourselves to the case of Poisson arrivals
and regularly varying service time distribution. In this case we are
able to obtain detailed tail asymptotics for the waiting time
distribution by applying Lemma~\ref{lem:tauber} for Laplace-Stieltjes
transforms (LST) of regularly varying distributions to an expression
of Le Gall~\cite{LeGall} for the waiting time LST in the $M/G/1$ queue
with ROS\@. In Section~\ref{rvprob}, we shall present an alternative
approach to the same result, viz., we shall work out the general tail
asymptotics of Section~\ref{global} for this case.

We consider an $M/G/1$ queue with arrival rate $\lambda =1/\alpha$ and service
time distribution $B(\cdot )$ with mean $\beta $ and LST
$\beta \{\cdot \}$. As before, the load of the
queue is $\rho \egaldef \lambda \beta <1$.

The LST of the waiting time distribution for the ROS discipline is
given by (see Le Gall~\cite{LeGall} or Cohen~\cite{Coh:2}, p.~439):
\begin{equation}
\EE  [\e^{-sW\ROS}]  =  
   1-\rho +\rho \beta ^{fw}\{s\}
   -\frac{\rho (1-\rho )}{\beta s}
      \int_{\mu \{s\}}^1\frac{\partial\Phi}{\partial z}(s,z)\Psi(s,z)\d z,
                                                          \label{eq:waitros}
\end{equation}
with
\begin{eqnarray}
\Phi (s,z) & \egaldef  & 
   (1-z)\frac{\beta\{s+\lambda (1-z)\}-\beta\{\lambda (1-z)\}}
             {z-\beta \{\lambda (1-z)\}},
\label{phisz}
\\
\Psi (s,z) & \egaldef  & 
   \exp \left[-\int_z^1
                   \frac{\d y}{y-\beta \{s+\lambda (1-y)\}}\right],
\label{psisz}
\end{eqnarray}
where $\beta^{fw}\{\cdot \}$ is the LST of the forward recurrence time of the
service time: 
\begin{equation}
\beta^{fw}\{s\}\egaldef \frac{1-\beta \{s\}}{\beta s},\label{eq:betar}
\end{equation}
and $\mu\{s\}$ is the LST of the busy period distribution, satisfying
the relation
\begin{equation}\label{eq:mu}
\mu \{s\}-\beta\bigl\{s+\lambda (1-\mu \{s\})\bigr\}=0.
\end{equation}

It is possible to rewrite~(\ref{eq:waitros}) in a simpler form using
integration by parts. Indeed
\[ 
  \int_{\mu\{s\}}^1\frac{\partial\Phi}{\partial z}(s,z)\Psi(s,z)\d z
  =\Bigl[\Phi(s,z)\Psi(s,z)\Bigr]_{\mu\{s\}}^1
   -\int_{\mu\{s\}}^1\frac{\Phi(s,z)\Psi(s,z)\d z}{z-\beta\{s+\lambda(1-z)\}},
\] 
and the simple relations
\[
  \Phi(s,\mu(s)) = 1-\mu\{s\},
  \quad\Phi (s,1)=\frac{1-\beta\{s\}}{1-\rho},
\]
\[
  \Psi (s,\mu\{s\})=0,
  \quad\Psi (s,1)=1,
\]
yield
\[ 
  \Bigl[\Phi(s,z)\Psi(s,z)\Bigr]_{\mu\{s\}}^1=\frac{1-\beta\{s\}}{1-\rho}.
\]

Using this relation and~(\ref{eq:betar}) in~(\ref{eq:waitros})
yields the following simpler alternative:
\begin{equation}
\EE [\e^{-sW\ROS}] 
   =  1+\frac{\rho(1-\rho)}{\beta s}
            \int_{\mu\{s\}}^1\widehat{\Phi}(s,z)\Psi(s,z)\d z,
                                               \label{eq:waitros2}
\end{equation}
with
\begin{eqnarray}
\widehat{\Phi}(s,z) 
   & \egaldef& \frac{\Phi(s,z)-\frac{\beta s}{1-\rho}}
                   {z-\beta\{s+\lambda(1-z)\}}\nonumber\\
 & = & \frac{1-z-\frac{\beta s}{1-\rho}}{z-\beta\{s+\lambda(1-z)\}}
       -\frac{1-z}{z-\beta\{\lambda(1-z)\}}.\nonumber
\end{eqnarray}

As before,
we write $f(x) \sim g(x)$ if
$f(x)/g(x) \rightarrow 1$ when $x \rightarrow \infty$,
and similarly we write $a(s) \sim b(s)$ if $a(s)/b(s) \rightarrow 1$
when $s \rightarrow 0$.
In this section
and the next 
we assume that the service requirement distribution
$B(\cdot)$ is regularly varying with index $-\nu$, $1<\nu<2$:
\begin{equation}
\PP(B>x) \sim \frac{C}{\Gamma(1-\nu)} x^{-\nu} L(x), ~~~ x \rightarrow \infty ,
\label{BRV}
\end{equation}
with $C$ a constant, $\Gamma(\cdot)$ the Gamma function and $L(\cdot)$
a slowly varying function at infinity, cf.~\cite{BGT87}.

Lemma~\ref{lem:tauber} in Appendix~\ref{app:classes} implies, in
combination with the assumption that~(\ref{BRV}) holds with $\nu \in
(1,2)$:
\begin{equation}
\beta\{s\}-1+\beta s\sim C s^{\nu}L(1/s),
  \textrm{ as }s\to0.\label{eq:beta:tail}
\end{equation}
In addition,
cf.\ (\ref{eq:betar}),
\begin{equation}
1-\beta^{fw}\{s\}\sim\frac{C}{\beta}s^{\nu-1}L(1/s),\textrm{ as }s\to0.
\label{betares}
\end{equation}
De Meyer and
Teugels~\cite{MeyTeu:1} have proven that the busy period distribution
of an $M/G/1$ queue with regularly varying service time distribution
is also regularly varying at infinity. More precisely:
\begin{equation}
\mu\{s\}-1+\frac{\beta s}{1-\rho}
  \sim \frac{C}{(1-\rho)^{\nu+1}}s^{\nu}L(1/s),
\textrm{ as }s\to0.\label{eq:mu:tail}
\end{equation}
We shall use the first-order behavior of $\mu\{s\}$ further on.


Our goal is to prove the following theorem.

\begin{theorem}\label{thm:mg1}
If the service time distribution in the $M/G/1$ queue
operating under the ROS discipline is regularly varying at infinity
with index $-\nu \in (-2,-1)$, then
the waiting time distribution is regularly varying at infinity
with index $1-\nu \in (-1,0)$.
More precisely, if (\ref{BRV}) holds then, as $x\to\infty$,
\begin{eqnarray*}
\PP(W\ROS>x) 
  &\sim& \frac{\rho}{1-\rho}h(\rho, \nu)\PP(B^{fw} > x)\\
  &\sim& \frac{\rho}{1-\rho}h(\rho, \nu)
         \frac{1}{\Gamma(2-\nu)}\frac{C}{\beta}x^{1-\nu} L(x),
\end{eqnarray*}
with
\begin{eqnarray}
h(\rho, \nu) &\egaldef& \int_0^1f(u,\rho,\nu)\d u , \label{311}\\ 
f(u,\rho,\nu) 
   &\egaldef& \frac{\rho}{1-\rho} \left(\frac{\rho u}{1-\rho}\right)^{\nu-1}
              (1-u)^{\frac{1}{1-\rho}}
              +\left(1+\frac{\rho u}{1-\rho}\right)^{\nu}
              (1-u)^{\frac{1}{1-\rho}-1}.
\label{312}
\end{eqnarray}
\end{theorem}

\begin{remark}
The following result for the waiting time distribution in the $M/G/1$
queue operating under the FCFS discipline is well-known (cf.\
Cohen~\cite{COH73} for the regularly varying case; see Pakes
\cite{Pakes75} for an extension to the larger class of subexponential
residual service time distributions): if (\ref{BRV}) holds, then
\begin{equation}
\PP(W\FCFS>x) \sim \frac{\rho}{1-\rho} \PP(B^{fw} > x), ~~~ x\to\infty.
\label{WFCFS}
\end{equation}
We can now conclude that
\begin{equation}\label{eq:wros:wfcfs}
\PP(W\ROS>x) \sim h(\rho,\nu)\PP(W\FCFS>x), ~~~ x \rightarrow \infty .
\end{equation} 
\end{remark}

\begin{remark}
It is actually possible to prove that $h(\rho,\nu)$ is always less
than $1$: indeed, the function $f(u,\rho,\nu)$ is strictly convex in
$\nu$ (as a sum of exponentials) and therefore
$h(\rho, \nu)$ is also strictly convex in $\nu$;
moreover, simple computations using integration by parts show that
\[
  \int_0^1f(u,\rho,1)\d u = \int_0^1f(u,\rho,2)\d u = 1,
\]
and therefore
\begin{equation}
  h(\rho, \nu) < 1, \ \ \mbox{for all }1<\nu<2,\ 0\leq\rho<1.
\end{equation}

Numerical computations with MAPLE suggest that $h(\rho,\nu)$ is
decreasing in $\rho$ and thus always larger than its limit when
$\rho\to1$, which is by simple arguments equal to $\Gamma(\nu)$. 
Unfortunately, we
have not been able to find a simple proof for this fact.
In Figure 1 we have plotted $h(\rho,\nu)$ for $0 < \rho <1$ and $1 <
\nu < 2$.
\begin{figure}\label{fig:h}
\begin{center}
\ifx\pdfoutput\undefined
  \includegraphics[angle=-90]{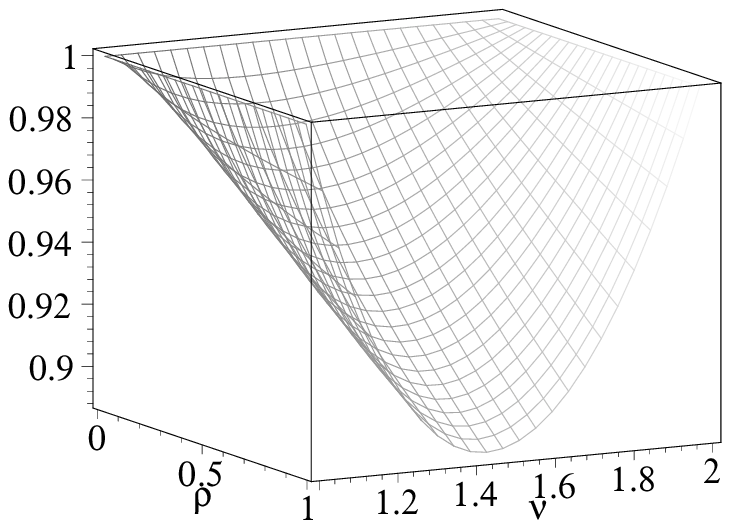}
\else
  \includegraphics{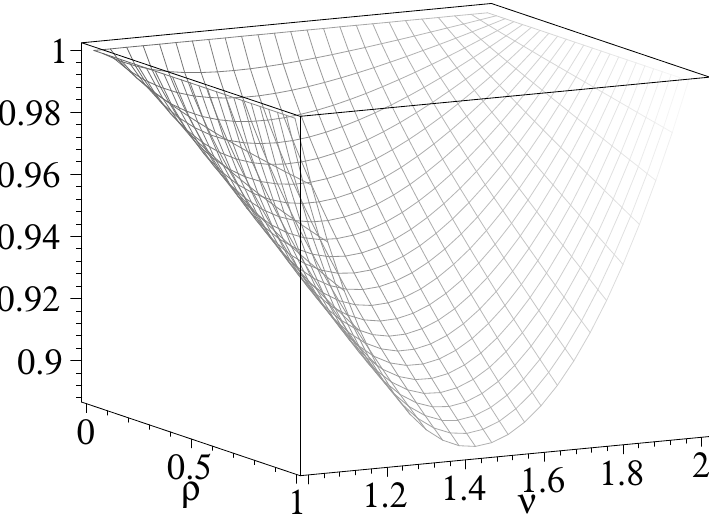}
\fi
\end{center}
\caption{A plot of the function $(\rho,\nu)\mapsto h(\nu,\rho)$ for
  $0\leq\rho\leq 1$ and $1\leq\nu\leq 2$. Note that the minimal value
  for $h$ seems to be $h(1,3/2)=\Gamma(3/2)\approx 0.88622\cdots$.}
\end{figure}

\noindent
It is interesting to observe that the tail behavior of $W_{ROS}$ and
$W_{FCFS}$ is so similar in the regularly varying case. This strongly
contrasts with the situation for the $M/M/1$ queue, where the purely
exponential waiting time tail for FCFS strongly deviates from the
$C_0x^{-5/6} {\rm e}^{-C_1 x -C_2 x^{1/3}}$ waiting time tail behavior
that was exposed by Flatto \cite{Flatto} for ROS.
\end{remark}

\begin{remark}\label{rem:fg}
The first part of $h(\rho,\nu)$ is a Beta function, and the second
part is a hypergeometric function (cf.\ Abramowitz and
Stegun~\cite{AS}). In particular,
\[
\int_0^1 u^{\nu-1} (1-u)^{\frac{1}{1-\rho}} {\rm d}u 
  = B\Bigl(\nu,\frac{1}{1-\rho}+1\Bigr)
  = \frac{\Gamma(\nu) \Gamma(\frac{1}{1-\rho}+1)}{\Gamma(\nu + \frac{1}{1-\rho}+1)} .
\]
Using partial integration and the above formula for Beta functions,
one gets the following form, which is useful for future comparisons
(see Section~\ref{rvprob}):
\begin{equation}
h(\rho,\nu) = 1-\rho + \int_0^1 g(u,\rho,\nu) {\rm d}u,
\label{eq:fgrel}
\end{equation}
with
\begin{equation}
g(u,\rho,\nu) 
   \egaldef  \left[
\frac{\rho\nu(1-u)-\rho}{u}
\left(
\frac{\rho u}{1-\rho}\right)^{\nu-1}
             +\rho \nu \left(1+\frac{\rho u}{1-\rho}\right)^{\nu-1} \right]
              (1-u)^{\frac{1}{1-\rho}}.
\nonumber
\end{equation}
\end{remark}

\begin{proof}[Proof of Theorem~\ref{thm:mg1}]
We shall prove the theorem by applying Lemma~\ref{lem:tauber} to an
expression for the LST of the waiting time distribution. So we need to
consider $\EE [\e^{-sW\ROS}]-1$ as $s\to0$. In order to do that, we use the
change of variable $1-z=\epsilon(s)u$, with
$\epsilon(s)\egaldef1-\mu\{s\}$, in (\ref{eq:waitros2}) to obtain
\[
\int_{\mu\{s\}}^1\widehat{\Phi}(s,z)\Psi(s,z)\d z
  =\epsilon(s)\int_0^1\widehat{\Phi}(s,1-\epsilon(s)u)
                             \Psi(s,1-\epsilon(s)u)\d u.
\]

Let us first evaluate the function $\Psi$. One can write
\[
\log\Psi(s,1-\epsilon(s)u)
  =-\int_0^1\frac{\epsilon(s)u\d v}
                     {1-\epsilon(s)uv-\beta\{s+\lambda\epsilon(s)uv\}},
\]
and, using~(\ref{eq:beta:tail}) and~(\ref{eq:mu:tail}),
\[
\lim_{s\to0}\frac{1}{\epsilon(s)}
            \Bigl[1-\epsilon(s)uv-\beta\{s+\lambda\epsilon(s)uv\}\Bigr]
  =(1-\rho)(1-uv),
\]
the limit being uniform in $u$ and $v$. Therefore the following
limit holds uniformly in $u$:
\begin{equation}
\lim_{s\to0}\Psi(s,1-\epsilon(s)u)
  =\exp\left[-\frac{1}{1-\rho}\int_0^1\frac{u\d v}{1-uv}\right]
  =(1-u)^{\frac{1}{1-\rho}}.\label{eq:limPsi}
\end{equation}

The evaluation of $\widehat{\Phi}$ is not difficult either. First note
that the denominators appearing in $\widehat{\Phi}(s,1-\epsilon(s)u)$
can be expressed in terms of $\beta^{fw}\{\cdot\}$ as follows:
\begin{eqnarray*}
\lefteqn{1-\epsilon(s)u-\beta\{s+\lambda\epsilon(s)u\}\, =}\quad\quad  &  & \\
 &  & -(1-\rho)\epsilon(s)u+\beta s
      -(\beta s+\rho\epsilon(s)u)(1-\beta^{fw}\{s+\lambda\epsilon(s)u\}), \\
\lefteqn{1-\epsilon(s)u-\beta\{\lambda\epsilon(s)u\}\, =}\quad\quad  &  & \\
 &  & -(1-\rho)\epsilon(s)u
      -\rho\epsilon(s)u(1-\beta^{fw}\{\lambda\epsilon(s)u\}).
\end{eqnarray*}
\noindent
Then write
\begin{eqnarray*}
\hspace{-2cm}
\lefteqn{\Bigl(\epsilon(s)u-\frac{\beta s}{1-\rho}\Bigr)
         (1-\epsilon(s)u-\beta\{\lambda\epsilon(s)u\})
         -\epsilon(s)u(1-\epsilon(s)u-\beta\{s+\lambda\epsilon(s)u\})}
\qquad\qquad  &  & \\
 & = & \epsilon(s)u
       \Bigl[(\beta s+\rho\epsilon(s)u)(1-\beta^{fw}\{s+\lambda\epsilon(s)u\})\\
 &  &  \phantom{\epsilon(s)u}
       +\Bigl(\frac{\rho\beta s}{1-\rho}-\rho\epsilon(s)u\Bigr)
        (1-\beta^{fw}\{\lambda\epsilon(s)u\})\Bigr],
\end{eqnarray*}
and finally, as $s\to0$,
\begin{eqnarray*}
\widehat{\Phi}(s,1-\epsilon(s)u)
  &\sim&  -\frac{1}{(1-\rho)(1-u)}
             \Bigl[\Bigl(1+\frac{\rho u}{1-\rho}\Bigr)
                   (1-\beta^{fw}\{s+\lambda\epsilon(s)u\})\\
 &  & \phantom{\frac{1}{(1-\rho)(1-u)}}
                    +\frac{\rho(1-u)}{1-\rho}
                     (1-\beta^{fw}\{\lambda\epsilon(s)u\})\Bigr].
\end{eqnarray*}

We take into account now the regular variation assumption~(\ref{BRV}),
which yields (\ref{betares})
so that
\begin{eqnarray*}
\widehat{\Phi}(s,1-\epsilon(s)u) 
  & \sim& -\frac{Cs^{\nu-1}}{\beta(1-\rho)(1-u)}\biggl\{
                \frac{\rho(1-u)}{1-\rho}
                \Bigl(\frac{\rho u}{1-\rho}\Bigr)^{\nu-1}
                L\Bigl(\frac{1}{\lambda\epsilon(s)u}\Bigr)\\
 &  & \phantom{-\frac{Cs^{\nu-1}}{\beta(1-\rho)(1-u)}}
                +\Bigl(1+\frac{\rho u}{1-\rho}\Bigr)^{\nu}
                 L\Bigl(\frac{1}{s+\lambda\epsilon(s)u}\Bigr)\biggr\}.
\end{eqnarray*}
Using Potter's Theorem (see~\cite{MeyTeu:1}, Theorem 1.5.6), setting
$\delta\egaldef(\nu-1)/2$, there exists $X>0$ such that, as long as
$s\leq 1/X$ and $\lambda\epsilon(s)\leq 1/X$,
\begin{eqnarray*}
L\Bigl(\frac{1}{\lambda\epsilon(s)u}\Bigr)
  &\leq & 2\max\Bigl[\Bigl(\frac{s}{\lambda\epsilon(s)u}\Bigr)^\delta,
                     \Bigl(\frac{s}{\lambda\epsilon(s)u}\Bigr)^{-\delta}\Bigr]
          L(1/s),\\
L\Bigl(\frac{1}{s+\lambda\epsilon(s)u}\Bigr)
  &\leq & 2\max\Bigl[\Bigl(\frac{1}{1+\frac{\lambda\epsilon(s)}{s}u}\Bigr)^\delta,
                     \Bigl(\frac{1}{1+\frac{\lambda\epsilon(s)}{s}u}\Bigr)^{-\delta}\Bigr]
         L(1/s).
\end{eqnarray*}

These bounds allow application of the Dominated Convergence Theorem to
\[
  \int_0^1 \frac{1}{s^{\nu-1}L(1/s)}\widehat{\Phi}(s,1-\epsilon(s)u)
                             \Psi(s,1-\epsilon(s)u)\d u,
\]
which yields
\begin{equation}
1-\EE [\e^{-sW\ROS}] 
   \sim  \frac{\lambda C}{1-\rho}
      \Bigl[\int_0^1 f(u,\rho,\nu) \d u\Bigr] s^{\nu-1}L(1/s)
   \textrm{, as }s\to0.
\label{esw}
\end{equation}
\noindent
Using Lemma~\ref{lem:tauber}, the theorem follows.
\end{proof}

\section{Agreement of results}
\label{rvprob}

While Sections~\ref{global} and~\ref{rvlst} use completely different
methods of proof, it is clear that the asymptotics for the $M/G/1$
queue with regularly varying service time distribution (as in
Theorem~\ref{thm:mg1}) has to be a mere consequence of
Theorem~\ref{thm:main}. This section shows that this is indeed true.
As a first step, we give another asymptotic expression for $\PP(W\ROS>x)$
which, while less intuitive than the rewriting proposed in
Section~\ref{global}, bears a strong similarity with
Theorem~\ref{thm:mg1}.

\begin{lemma}\label{lem:main2}
In the GI/G/1 ROS queue with $\st\in{\cal L}\bigcap{\cal D}$, we have
\begin{eqnarray*}
\PP(W\ROS>x) &\sim&
  \rho\PP(B^{fw}>x)\\
  &&\mbox{}+\int_0^1 
     \biggl\{\frac{1}{\alpha(1-\rho)}
         \EE \Bigl[B\ind{\frac{x(1-\rho)}{\rho u+1-\rho}< B 
                   \leq \frac{x(1-\rho)}{\rho u}}\Bigr]\\
  &&\phantom{+\int_0^1\biggl(}
       + \frac{x}{\alpha u^2}\PP\Bigl(B>\frac{x(1-\rho)}{\rho u}\Bigr)
         (1-u)^{\frac{1}{1-\rho}}\,\dd u \biggr\}.
\end{eqnarray*}
\end{lemma}

\begin{proof}
To simplify the computations, assume that $B$ has a density function $b$.
This is, however, not really needed for the result.

The first term in Theorem~\ref{thm:main} coincides with that in the expression above.
Next, consider the second and third terms in Theorem~\ref{thm:main}
together, using the changes of variable $z\mapsto u=(x+v\alpha-z)(1-\rho
)/(\rho z)$ and $v\mapsto t=(x+v\alpha)(1-\rho)/(\rho u+1-\rho)$.
\begin{eqnarray*}
\lefteqn{\int_{0}^{cx}\dd v
\int_{(v\alpha+x)(1-\rho)}^{v\alpha+x} \dd z\,b(z)
\left(
1-\frac{(x+v\alpha-z)(1-\rho)}{\rho z} \right)^{\frac{1}{1-\rho}}}\qquad\\
\lefteqn{\quad+\int_{cx}^{\infty}\dd v
\int_{v\alpha}^{v\alpha+x} \dd z\,b(z)
\left(
1-\frac{(x+v\alpha-z)(1-\rho)}{\rho z} \right)^{\frac{1}{1-\rho}}}\qquad
\\
&=& \int_{0}^{\infty}\dd v
\int_{(v\alpha+x)(1-\rho)}^{v\alpha+x} \dd z\,b(z)
\left(
1-\frac{(x+v\alpha-z)(1-\rho)}{\rho z} \right)^{\frac{1}{1-\rho}}\\
&& \mbox{}-\int_{cx}^{\infty}\dd v
\int_{(v\alpha+x)(1-\rho)}^{v\alpha} \dd z\,b(z)
\left(
1-\frac{(x+v\alpha-z)(1-\rho)}{\rho z} \right)^{\frac{1}{1-\rho}}\\
&=& \int_{0}^{\infty}\dd v
\int_0^1 \dd u\, b\left(\frac{(x+v\alpha)(1-\rho)}{\rho u+1-\rho}\right)
 \frac{\rho(x+v\alpha)(1-\rho)}{(\rho u+1-\rho)^2}(1-u)^{\frac{1}{1-\rho}}\\
&&\mbox{}-\int_{cx}^{\infty}\dd v
\int_0^{\frac{cx}{v}}\dd u\, b\left(\frac{(x+v\alpha)(1-\rho)}{\rho u+1-\rho}\right)
 \frac{\rho(x+v\alpha)(1-\rho)}{(\rho u+1-\rho)^2}(1-u)^{\frac{1}{1-\rho}}\\
&=& \int_0^1 \dd u\, (1-u)^{\frac{1}{1-\rho}}
        \int_{0}^{\frac{cx}{u}}\dd v\,
           b\left(\frac{(x+v\alpha)(1-\rho)}{\rho u+1-\rho}\right)
           \frac{\rho(x+v\alpha)(1-\rho)}{(\rho u+1-\rho)^2}\\
&=& \frac{1}{\alpha(1-\rho)}\int_0^1 \dd u\, (1-u)^{\frac{1}{1-\rho}}
       \int_{\frac{x(1-\rho)}{\rho u+1-\rho}}
           ^{\frac{x(1-\rho)}{\rho u}}
                    \dd t\,tb(t).
\end{eqnarray*}
Finally, focus on the last term of Theorem~\ref{thm:main}. We use the
changes of variables $z\mapsto w=z- v\alpha(1-\rho)$ and then $w\mapsto
u=x(1-\rho)/w$.
\begin{eqnarray}
\lefteqn{\int_{cx}^{\infty}\dd v
  \int_{(v\alpha+x)(1-\rho)}^{v\alpha} \dd z\, b(z)
  \left(1-\frac{x(1-\rho)}{z-v\alpha(1-\rho)}\right)^{\frac{1}{1-\rho}}}
       \nonumber\qquad\\
&=& 
  \int_{cx}^{\infty}\dd v
    \int_{x(1-\rho)}^{\rho v\alpha} \dd w\, b\bigl(w+v\alpha(1-\rho)\bigr)
    \left(1-\frac{x(1-\rho)}{w}\right)^{\frac{1}{1-\rho}}\nonumber\\
&=&
  \int_{x(1-\rho)}^{\infty} 
      \dd w\left(1-\frac{x(1-\rho)}{w}\right)^{\frac{1}{1-\rho}}
      \int_{\frac{w}{\rho\alpha}}^\infty \dd v\, b\bigl(w+va(1-\rho)\bigr)\nonumber\\
&=& 
  \frac{1}{\alpha(1-\rho)}\int_{x(1-\rho)}^{\infty} 
      \dd w\left(1-\frac{x(1-\rho)}{w}\right)^{\frac{1}{1-\rho}}
      \PP\left(B>\frac{w}{\rho}\right)\nonumber\\
&=& 
  \frac{x}{\alpha}\int_0^1\frac{1}{u^2}
        \PP\left(B>\frac{x(1-\rho)}{\rho u}\right)(1-u)^{\frac{1}{1-\rho}}du.
    \label{eq:lasterm}
\end{eqnarray}
The proof of the lemma is completed by collecting the terms.
\end{proof}

\begin{remark}
It is interesting to note that~(\ref{eq:lasterm}) can be rewritten as
follows:
\begin{eqnarray*}
\lefteqn{\frac{x}{\alpha}\int_0^1\frac{1}{u^2}
  \PP\left(B>\frac{x(1-\rho)}{\rho u}\right)(1-u)^{\frac{1}{1-\rho}}du}\qquad\\
 &=& \frac{x}{\alpha}\int_0^1\frac{1}{u^2}
  \PP\left(B>\frac{x(1-\rho)}{\rho u}\right)
  \PP\Bigl(W\ROS^*>\frac{\beta u}{1-\rho}\Bigr)du\\
 &=& \frac{x}{\alpha}\frac{\rho}{x(1-\rho)} 
      \int_{\frac{x(1-\rho)}{\rho}}^\infty
          \PP(B>z)\PP\Bigl(W\ROS^*>\frac{\alpha x}{z}\Bigr)dz\\
 &=& \frac{\rho^2}{1-\rho}
       \PP\Bigl(B^{fw}>\frac{x(1-\rho)}{\rho}, B^{fw}W\ROS^*>\alpha x\Bigr).
\end{eqnarray*}
%
\end{remark}
In the case of Poisson arrivals and regularly varying service times,
assuming again that
\[
    \PP\left(B>x\right)\sim\frac{C}{\Gamma(1-\nu)}x^{-\nu}L(x),
\]
and recalling that $\lambda=1/\alpha$, it is easy to go from the
expression in Lemma~\ref{lem:main2} to the expression
\begin{eqnarray*}
  \PP(W\ROS>x) &\sim &
 \frac{C L(x) x^{1-\nu}}{\alpha(1-\rho)\Gamma(2-\nu)}
       \left[1-\rho+\int_{0}^1 g(u,\rho,\nu)\dd u\right].
\end{eqnarray*}

Using~(\ref{eq:fgrel}) it is seen that this corresponds to
Theorem~\ref{thm:mg1}.

\section{Heavy-traffic limit for the waiting time distribution}
\label{heavytraffic}

In this section, we consider the $M/G/1$ queue with general service
time distribution. We are interested in the heavy-traffic case:
$\rho\to1$. The main result is to find a sequence $\Delta(\rho)$
which tends to $0$ as $\rho\to1$ such that $\EE  [\e^{-\omega
\Delta(\rho)W\ROS}]$ tends to a proper limit as $\rho\to1$. 
This way we
shall be able to retrieve a result of Kingman~\cite{Kin:2} for the
case of finite service time variance, as well as derive a new result
for the case of regularly varying service time distribution with
infinite variance.

The following lemma will be useful in the sequel.

\begin{lemma}\label{lem:Psi}
The following bound holds for all $s>0$ and $0\leq z\leq 1$:
\begin{equation}\label{eq:boundPsi}
  \Psi(s,z)\leq\exp\Bigl[-\frac{1-z}{\beta s}\Bigr].
\end{equation}
Moreover, for any $t>0$, 
\begin{equation}\label{eq:limPsi2}
  \lim_{(\rho,s)\to(1,0)}\Psi(s,1-\beta st)=\e^{-t}.
\end{equation}
\end{lemma}
\begin{proof}
Using the inequality $\beta\{s\}\geq1-\beta s$, one finds
\[
  y-\beta\{s+\lambda(1-y)\}\leq y-1+\beta s+\rho(1-y)\leq\beta s,
\]
and~(\ref{eq:boundPsi}) follows from the definition~(\ref{psisz}),
since
\[
  -\log\Psi(s,z)\geq \int_z^1 \frac{dy}{\beta s}=\frac{1-z}{\beta s}.
\]

The proof of~(\ref{eq:limPsi2}) follows a similar argument: indeed, 
\[
-\log\Psi(s,1-\beta st) 
  = \int_{1-\beta st}^1\frac{\d y}{y-\beta\{s+\lambda(1-y)\}}
  =  \int_{0}^1\frac{\beta st\d v}{1-\beta stv -\beta\{s+\rho stv\}} ,
\]
and, as $s\to0$,
\[
  1-\beta stv -\beta\{s+\rho stv\} \sim \beta s(1-(1-\rho)tv).
\]
This yields~(\ref{eq:limPsi2}) and concludes the proof of the lemma.
\end{proof}

\noindent
The LST of the steady-state waiting time distribution
under FCFS is given by (cf.~\cite{Coh:2}, p.\ 255): for any $s \geq 0$, 
\[
 \EE [\e^{-sW\FCFS}] = \frac{1-\rho}{1-\rho\beta^{fw}\{s\}}.
\]

The following lemma illustrates the relation between $W\ROS$ and $W\FCFS$
in heavy traffic.

\begin{lemma}\label{lem:heavytraffic}
Assume that $\beta<\infty$ and that there exists $\Delta(\rho)>0$ which can serve as a proper
scaling for $W\FCFS$: for any $\omega >0$,
\begin{equation}\label{eq:scalefcfs}
  \lim_{\rho\to1} \EE [\e^{-\omega \Delta(\rho)W\FCFS}] 
          = \EE [\e^{-\omega \widehat{W}\FCFS}],
\end{equation}
where $\widehat{W}\FCFS$ is a non-negative random variable. 
Then $\Delta(\rho)$ is a proper scaling for $W\ROS$: for any $\omega >0$,
\[
 \lim_{\rho\to1} \EE [\e^{-\omega \Delta(\rho)W\ROS}] 
   =\int_{0}^{\infty} \EE [\e^{-\omega t\widehat{W}\FCFS}]\e^{-t}\d t
   \egaldef\EE [\e^{-\omega \widehat W\ROS}].
\]
\end{lemma}

\begin{proof}
The starting point of the proof is Equation~(\ref{eq:waitros2}).
First, using integration by parts,
\[
\int_{\mu\{s\}}^1\frac{1-z-\frac{\beta s}{1-\rho}}{z-\beta\{s+\lambda(1-z)\}}
                             \Psi(s,z)\d z 
  = -\frac{\beta s}{1-\rho}+\int_{\mu\{s\}}^1\Psi(s,z)\d z ,
\]
and the integral above can be bounded using~(\ref{eq:boundPsi}) as
\[
 \int_{\mu\{s\}}^1\Psi(s,z)\d z 
   \leq \int_{\mu\{s\}}^1\exp\Bigl[-\frac{1-z}{\beta s}\Bigr]\d z
   \leq \beta s.
\]
The second part of~(\ref{eq:waitros2}) can be expressed in terms of
$W\FCFS$: 
\begin{eqnarray*}
 \lefteqn{\int_{\mu\{s\}}^1 \frac{1-z}{z-\beta\{\lambda(1-z)\}}
                             \Psi(s,z)\d z}\qquad\qquad\\
  &=& -\int_{\mu\{s\}}^1 \frac{1}{1-\rho}\EE [\e^{-\lambda(1-z)W\FCFS}]
                             \Psi(s,z)\d z\\
  &=& -\frac{\beta s}{1-\rho}\int_0^{\frac{\epsilon(s)}{\beta s}} 
           \EE [\e^{-\rho stW\FCFS}]\Psi(s,1-\beta st)\d t.      
\end{eqnarray*}
Plugging these two relations into~(\ref{eq:waitros2}) yields 
\begin{equation}\label{eq:waitrosfcfs}
\EE  [\e^{-sW\ROS}] = \rho \int_0^{\frac{\epsilon(s)}{\beta s}} 
           \EE [\e^{-\rho stW\FCFS}]\Psi(s,1-\beta st)\d t
    +O(1-\rho),
\end{equation}
uniformly in $s>0$.
We now show that
\begin{equation}\label{eq:limeps}
  \lim_{\rho\to1} \frac{\epsilon(\omega \Delta(\rho))}{\Delta(\rho)}
  = +\infty,
\end{equation}
with, as before, $\epsilon(s)\egaldef1-\mu\{s\}$ and $\mu\{s\}$ is the LST of the busy period distribution.
Let $\mu_{\bar{\rho}}(s)$ be the LST of the busy period of the $M/G/1$ with service time distributions $B(\cdot)$ and load ${\bar{\rho}}$.
Obviously, for $\rho\geq{\bar{\rho}}$ we have $\mu(s)\leq\mu_{\bar{\rho}}(s)$, for all $s>0$.
Hence,
\[
  \liminf_{\rho\to1} \frac{\epsilon(\omega \Delta(\rho))}{\Delta(\rho)}
  \geq  \lim_{\rho\to1} \frac{1-\mu_{\bar{\rho}}(\omega \Delta(\rho))}{\Delta(\rho)}
  = \frac {\beta\omega}{1-{\bar{\rho}}}.
\]
This is true for any fixed ${\bar{\rho}}\in(0,1)$.
Letting ${\bar{\rho}}$ pass to~1 we obtain~(\ref{eq:limeps}).

Finally, under Assumption~(\ref{eq:scalefcfs}), Lemma~\ref{lem:Psi}
allows to apply the Dominated Convergence Theorem
to~(\ref{eq:waitrosfcfs}), and the lemma is proved.
\end{proof}

Using Feller's continuity theorem, Lemma~\ref{lem:heavytraffic} can be
rewritten in a more compelling way.

\begin{corollary}\label{cor:heavytraffic}
Assume that there exists $\Delta(\rho)>0$, and a random variable
$\widehat W\FCFS$, such that the following limit holds in distribution:
\[
  \lim_{\rho\to1}\Delta(\rho)W\FCFS=\widehat W\FCFS.
\]
Then, in distribution,
\begin{equation}
  \lim_{\rho\to1} \Delta(\rho)W\ROS = Y\widehat W\FCFS \egaldef \widehat W\ROS ,
\label{expscale}
\end{equation}
where $Y$ is an exponential random variable with mean $1$, independent
of $\widehat W\FCFS$.
\end{corollary}
\begin{remark}
In view of the PASTA property and the fact that the workload is the
same under FCFS and ROS, (\ref{expscale}) states that the scaled
waiting time $\widehat W\ROS$ equals in distribution the product of
the unit exponential $Y$ and the scaled workload $\widehat V\ROS\eqdist\widehat W\FCFS$.
Put differently,
$\PP(\hat{W}\ROS >x| \hat{V}\ROS =y)$ $= \PP(Y>x/y) = {\rm e}^{-x/y}$.
The latter result
might be intuitively understood by referring to a \emph{snapshot
principle}. Consider a tagged customer. If it is not being elected for
service at the end of some ROS services, then in the heavy-traffic
limit the remaining workload that it sees has not changed. The
randomness (`memoryless') property of ROS then implies that the
remaining waiting time of the tagged customer has the same
distribution as before.
\end{remark}

As a first application of Lemma~\ref{lem:heavytraffic}, consider the
case where the service time distribution has a finite second moment.
The following theorem extends a result of Kingman~\cite{Kin:2}, where
it was additionally assumed that $\beta\{s\}$ exists for some $s<0$.

\begin{theorem}\label{thm:heavytraffic:light}
Assume that the variance $\sigma^2$ of the service time is finite and
let
\begin{equation}\label{eq:deltalight}
 \Delta(\rho)=\frac{\lambda(1-\rho)}{1+\frac{1}{2}\lambda^2\sigma^2}.
\end{equation}

Then, for any $\omega >0$,
\[
 \lim_{\rho\to1} \EE [\e^{-\omega \Delta(\rho)W\ROS}] 
   =\int_{0}^{\infty} \frac{\e^{-t}}{1+\omega t}\d t.
\]
\end{theorem}

\begin{proof}
It is known from a classical result of Kingman~\cite{Kin:1} that,
with $\Delta(\rho)$ defined by~(\ref{eq:deltalight}),
\[
  \lim_{\rho\to1} \EE [\e^{-\omega \Delta(\rho)W\FCFS}] 
          = \frac{1}{1+\omega},
\]
so that we may apply Lemma~\ref{lem:heavytraffic}.
\end{proof}

\begin{remark}
Thus, when the service times have a finite variance, $\widehat W\FCFS$ has an exponential distribution, so that $\widehat W\ROS$ is the product of two independent exponentials with unit mean.
Letting $K_1(\cdot)$ be the modified Bessel function of the second kind, simple calculations show that
\begin{eqnarray*}
 \PP(\widehat W\ROS>x) 
    &=& \int_0^\infty \PP(t\widehat W\FCFS>x)\e^{-t}\d t\\
    &=& \int_0^\infty \exp\Bigl[-\frac{x}{t}-t\Bigr]\d t\\
    &=& 2\sqrt{x}K_1(2\sqrt{x}),
\end{eqnarray*}
which coincides with Theorem~6 of~\cite{Kin:2}.
\end{remark}

In the case of a service time distribution with regularly varying tail
and infinite variance
a similar, but new, result can be obtained from
results of Boxma and Cohen~\cite{BC} for FCFS.

\begin{theorem}\label{thm:heavytraffic:heavy}
Under Assumption~(\ref{eq:beta:tail}),
with $1 < \nu < 2$,
let $\Delta(\rho)$ be the unique root
of the equation
\begin{equation}\label{eq:deltaheavy}
 \lambda C x^{\nu-1}L(x)=1-\rho,\ x>0,
\end{equation}
such that $\Delta(\rho) \downarrow 0$ for $\rho \uparrow 1$.
Then
\[
  \lim_{\rho\to1} \EE [\e^{-\omega \Delta(\rho)W\ROS}] 
    = \int_0^\infty \frac{\e^{-t}}{1+(\omega t)^{\nu-1}}\d t.
\]
\end{theorem}

\begin{proof}
It has been proved in~\cite{BC} that
$\Delta(\rho)$ exists and that for any $\omega >0$, under
Assumption~(\ref{eq:beta:tail}),
\[
  \lim_{\rho\to1} \EE [\e^{-\omega \Delta(\rho)W\FCFS}] = \frac{1}{1+\omega^{\nu-1}}.
\]
The theorem now follows from Lemma~\ref{lem:heavytraffic}.
\end{proof}

\section*{Appendices}
\appendix
\section{Classes of distributions
\label{app:classes}}

{\bf Definitions and properties}
\\
We say that a random variable belongs to a certain class if its
distribution function belongs to that class.
\begin{enumerate}
\item \label{1}
A cdf $F$ belongs to the class ${\cal L}$
of long-tailed distributions if
there exists a $y>0$ (or, equivalently, for all $y>0$) such that,
as $x\to\infty$,
$$ 
\frac{\overline{F}(x+y)}{\overline{F}(x)} \rightarrow 1.
$$
\begin{enumerate}
\item \label{1.1}
If $F\in {\cal L}$, $c>0$, and $G$ is another distribution such that
$\overline{G}(x) \sim c\overline{F}(x)$ as $x\to\infty$, then
$G\in {\cal L}$.
\item \label{1.2}
If $F\in {\cal L}$ and  
$m^+ \equiv m^+(F) = \int_0^{\infty} \overline{F}(t) dt$
is finite, then the integrated tail distribution
$F^I$ belongs to ${\cal L}$ too, but the converse is not true,
in general. Here
$$
F^I(x) = \max \left( 0;  1 - \int_x^{\infty}\overline{F}(t) dt
\right).
$$ 
\item \label{1.3}
$F^I \in {\cal L}$ if and only if $\overline{F}(x) =
o(\overline{F}^I(x))$ as $x\to\infty$.
\end{enumerate}
\item \label{2}
A cdf $F$ belongs to the class ${\cal RV}$
of regularly varying distributions if there exists
a $\nu >0$ such that
$$
\overline{F}(x) = x^{-\nu} L(x),
$$
where $L(x)$ is a slowly varying (at infinity) function.
\item\label{3}
A cdf $F$ belongs to the class ${\cal D}$ if
$$
\inf_{x\geq 0}
\frac{\overline{F}(2 x)}{\overline{F}(x)} > 0 .
$$
\begin{enumerate}
\item\label{3.1}
If $F\in {\cal D}$ and $m^+ < \infty$, then $F^I \in
{\cal D}$.
\end{enumerate}
\item \label{4}
A cdf $F$ belongs to the class ${\cal IRV}$
of intermediate regularly varying distributions if
$$
\lim_{\zeta \downarrow 1} \liminf_{x\to\infty}
\frac{\overline{F}(\zeta x)}{
\overline{F}(x)} = 1.
$$
\item\label{5}
A cdf $F$ on the positive  half-line belongs to the class ${\cal S}$
of subexponential distributions if
$$
\int_0^x F(dt) \overline{F}(x-t) \sim \overline{F}(x)
\quad \mbox{as} \quad x\to\infty .
$$
\begin{enumerate}
\item
A cdf $F$ on the real line belongs to ${\cal S}$ if $F(x) {\bf
  I}(x\geq 0)$ belongs to ${\cal S}$.
\end{enumerate}
\item\label{6}
A cdf $F$  belongs to the class ${\cal S}^{*}$ if $m^+$
is finite and
$$
\int_0^x \overline{F}(y) \overline{F}(x-y) dy \sim
2m^+ \overline{F}(x) \quad \mbox{as} \quad x\to\infty .
$$
\item\label{7}
Relations
\begin{enumerate}
\item\label{7.1}
\cite[p.~50]{EKM}
${\cal RV} \subset {\cal IRV} \subset {\cal L}
\bigcap {\cal D} \subset {\cal S}$.
\item\label{7.2}
\cite{Kluppelberg88}
If $F\in {\cal S}^{*}$, then
$F\in {\cal S}$ and $F^I\in {\cal S}$.
\item\label{7.3}
\cite{Kluppelberg88}
If 
$F \in {\cal L} \bigcap {\cal D}$  
and if  $m^+$
is finite,
 then $F\in {\cal S}^{*}$.
\item\label{7.4}
If $F\in {\cal D}$ and $F$ has an eventually non-increasing density, then $F\in {\cal IRV}$. Indeed, put
$$
K= \sup_{x\geq 0} \frac{\overline{F}(x/2)}{\overline{F}(x)}.
$$
Then, for $\zeta >1$ and for all sufficiently large $x$,
\begin{eqnarray*}
&&
\frac{\overline{F}( x)-\overline{F}(\zeta x)}{\overline{F}(x)}
=
\frac{\overline{F}(x/2)}{\overline{F}(x)}
\cdot
\frac{\overline{F}(x)-\overline{F}(\zeta x)}{\overline{F}(x/2)} \\
&&\leq 
K \frac{\int_x^{\zeta x} f(t) dt}{\int_{x/2}^{x} f(t) dt}
\leq K\frac{(\zeta -1)x f(x)}{\frac{1}{2}x f(x)} 
= 2 K(\zeta -1) 
\to 0
\end{eqnarray*}
as $\zeta \downarrow 1$.
\item\label{7.5}
It follows from Properties~(\ref{3.1}) and~(\ref{7.4}) that if $F\in {\cal D}$ and 
$m^+<\infty$, then $F^I \in {\cal IRV}$.
\end{enumerate}
\end{enumerate}

\bigskip
\noindent
There exists a very useful relation between the tail behavior of a
regularly varying probability distribution and the behavior of its LST
near the origin. That relation often enables one to conclude from the
form of the LST of a distribution, that the distribution itself is
regularly varying at infinity. We present this relation in
Lemma~\ref{lem:tauber} below. We use this in Section~\ref{rvlst} to
prove that the waiting time distribution in the $M/G/1$ queue under
the ROS discipline is regularly varying at infinity if the service
time distribution is regularly varying at infinity.

Let $F(\cdot)$ be the distribution of a non-negative random variable,
with LST $\phi\{s\}$ and finite first $n$ moments $\mu_1, \dots, \mu_n$
(and $\mu_0=1$). Define
\[
\phi_n\{s\} ~\egaldef~
(-1)^{n+1} \Bigl[ \phi\{s\} - \sum_{j=0}^n \mu_j \frac{(-s)^j}{j!}\Bigr].
\]
\begin{lemma}
Let $n<\nu<n+1$, $C \geq 0$. The following statements are equivalent:
\[
\phi_n\{s\} = (C+ o (1)) s^{\nu} L(1/s), ~~~ s \downarrow 0,
~~ s ~ {\rm real},
\]
\[
1-F(x) = (C + o (1)) \frac{(-1)^n}{\Gamma(1-\nu)} x^{-\nu} L(x),
~~~ x \rightarrow \infty .
\]
\label{lem:tauber}
\end{lemma}
\noindent
The case $C>0$ is due to Bingham and Doney~\cite{BD}.
The case $C=0$ was first obtained by Vincent Dumas,
and is treated in~\cite{BD982}, Lemma~2.2.
The case of an integer~$\nu$ is more complicated;
see Theorem 8.1.6 and Chapter~3 of~\cite{BGT87}.

\section{Proof of Theorem~\ref{thm:srp}}
\label{app:srp}

Note that the distribution of the residual service time of the customer in service is the same for all non-preemptive and non-idling service disciplines.
We may therefore concentrate on the FCFS discipline.

As before, $V_{-n}$ denotes the amount of work in the system upon
arrival 
of customer~$-n$ and~$T_{-n}$ is the time between arrival of
customer~$-n$ 
and time~0 (which is the arrival time of customer~0).
In the sequel the random variable $V$ has the stationary workload distribution.

{\it Lower bound.}
For any $\varepsilon \in (0,1)$, choose $K_1>0$ such that
$\PP (V > K_1 ) \leq \varepsilon $. Then choose an integer $n>0$
such that $\PP (T_{-n} >K_1 ) \geq 1-\varepsilon $. Third, choose
$K_2>K_1$ such that $\PP (T_{-n} \in (K_1,K_2 ))\geq 1-2\varepsilon $.
Then, as $x\to\infty$,
\begin{eqnarray*}
\PP (B^{rp}>x) &\geq &
\sum_{m=n}^{\infty} \PP (V_{-m} \leq K_1, T_{-n} \in (K_1,K_2),
B_{-m} > x+T_{-m})\\
&\geq &
\sum_{m=n}^{\infty} \PP (V_{-m} \leq K_1, T_{-n} \in (K_1,K_2),
B_{-m} > x+T_{-m}-T_{-n}+K_2)\\
&\geq &
(1-2\varepsilon )(1-\varepsilon )
\sum_{l=0}^{\infty} \PP (B_0 > x+K_2 + T_{-l}) \\
&\geq &
(1-2\varepsilon )(1-\varepsilon )
\rho \PP (B^{fw}>x+K_2) \\
&\sim &
(1-2\varepsilon )(1-\varepsilon )
\rho \PP (B^{fw}>x).
\end{eqnarray*}
Letting $\varepsilon \to 0$, we get the correct lower bound.
Since $B\in {\cal D}$, the lower bound is of order
$O(\PP (Z^{fw} >cx ))$ for any positive $c$.

{\it Upper bound.}
Let $y$ be a positive number and $\eta (y) = \min \{ n\geq 1 ~:~ 
\sum_{i=1}^n B_i > y\}$, $\chi (y) = \sum_1^{\eta (y)} B_i - y$.
Since $E B$ is finite, it follows from basic renewal theory
that the family of distributions of random
variables $\{ \chi (y), ~ y>0\}$ is tight, i.e.
$u(x) \equiv  \sup_{y>0} \PP (\chi (y)>x) \to 0$ as  $x\to\infty$.

Since $\srp\leq\brp$ almost surely, we have by
Corollary~\ref{cor:bpres} and by the lower bound obtained,
\bearno
    \PP(\srp>x)
    &=& \PP\left(\srp>x, \brp>x\right)
    \\
    &= &  (1+o(1))
    \PP (\srp>x, \brp>x, \bigcup_{m=1}^{\infty}\left\{ {\st}_{-m} >
      (x+m\alpha ) 
(1-\rho )\right\})
    \\
 &= & (1+o(1)) 
\sum_{m=1}^{\infty}
   \PP (\srp>x, Z^{rp}>x, {\st}_{-m} > 
(x+m\alpha ) (1-\rho ))
.
\enarno
Denote by $f_m(x)$ the $m$-th term in the latter sum. For any
$\varepsilon \in (0,1)$, choose $K>0$ such that $\PP (V>K)\leq
\varepsilon$.
Then
\begin{eqnarray*}
f_m(x) &\leq & \PP (V_{-m}\leq K, B_{-m}>x+T_{-m}-K) \\
&+&
\PP (V_{-m} >K, B_{-m} > (x+m\alpha )(1-\rho ))\\
&+&
\PP (B^{rp}>x, V_{-m}\leq K, B_{-m} \in
((x+m\alpha )(1-\rho )-K, x+T_{-m}-K); ~\exists 1\leq l \leq m ~:~
V_{-l}=0)\\
&+&
\EE
\int_{(x+m\alpha )(1-\rho )}^{x+T_{-m}-0}
\PP (V_{-m}\leq K, V_{-m}+B_{-m} \in dt)
\PP (\chi (x+T_{-m}-t)>x ~|~ T_{-m}) \\
&\equiv &
f_{m,1}(x) + f_{m,2}(x) + f_{m,3}(x) + f_{m,4}(x).
\end{eqnarray*}
Here
\begin{eqnarray*}
f_{m,1}(x) &\leq & \PP (B_{-m} > x+T_{-m}-K);\\
f_{m,2}(x) &\leq & \varepsilon \PP (B_{-m} > (x+m\alpha )(1-\rho ));\\
f_{m,3}(x) &\leq &
\sum_{l=1}^{m-1} \PP (B_{-m}>(x+m\alpha )(1-\rho ), V_{-l}=0,
V_{-j}>0 ~~ \forall ~ l<j<m) \PP (B^{rp}>x ~|~ V_{-l}=0)\\
&\leq &
\frac{\PP (Z^{rp}>x)}{\PP(V=0)} \PP (B_{-m} > (x+m\alpha )(1-\rho
));\\
f_{m,4}(x) &\leq & u(x) \PP (B_{-m} > (x+m\alpha )(1-\rho )).
\end{eqnarray*}
Since $u(x)\to 0$ and $\PP (Z^{rp}>x)\to 0$ as $x\to\infty$,
$$
\sum_{m=1}^{\infty} (f_{m,2}(x)+f_{m,3}(x)+f_{m,4}(x)) 
\leq
(1+o(1)) \varepsilon \PP(Z^{rp}>x).
$$
Further,
$$
\sum_{m=1}^{\infty} f_{m,1}(x) \leq
(1+o(1)) \rho \PP(B^{fw}>x-K) 
\sim 
\rho \PP(B^{fw}>x).
$$
Letting $\varepsilon \downarrow 0$, the proof is completed.
\qed

\section{Proof of Theorem~\ref{thm:main}}
\label{app:main}
We focus on the waiting time $W\ROS$ of customer~0 arriving at time~0.
Our proof consists of three main parts, each corresponding to a typical scenario in which the large waiting time arises.
The intuition behind these typical scenarios was discussed below Theorem~\ref{thm:main} (it is convenient to treat the two ``middle terms'' as one scenario).

Before proceeding, we note that the distribution of the waiting time of customer~0 is not affected if we choose to use the FCFS discipline before time~0 and the ROS discipline after time~0.
Thus, $W\ROS\eqdist W\ROS'$, where $W\ROS'$ denotes the waiting time of customer~0 under the modified service discipline.

The starting point of the proof is~(\ref{any}) in Corollary~\ref{cor:non-idling}, which we repeat for convenience (the modified service discipline is non-idling and non-preemptive):
\[
\PP (W\ROS'>x)
\sim
\sum_{m=1}^{\infty}
\PP \left(
W\ROS'>x, {\st}_{-m} > (
x+m\alpha)
(1-\rho ) \right),
\qquad
x\rightarrow\infty.
\]
In the verbal discussion, we interpret this relation as follows:
there is one ``large customer'', i.e., customer~$-m$ for which $B_{-m}>(x+m\alpha)(1-\rho)$, that causes the large waiting time of customer~0.
Note that any scenario in which the service of this large customer did not start before time~0 may be neglected (i.e., is of the order $o(\PP(\sfw>x))$):
\bearno
&&\sum_{m=1}^{\infty}\PP \left(W\ROS'>x, {\st}_{-m} > (x+m\alpha)(1-\rho ), T_{-m}\leq V_{-m} \right) \\
&&\leq\,\sum_{m=1}^{\infty}\PP \left({\st}_{-m} > (x+m\alpha)(1-\rho ), T_{-m}\leq V_{-m} \right) \\
&&\leq\,\sum_{m=1}^{\infty}\PP \left( {\st}_{-m} > (x+m\alpha)(1-\rho ), T_{-m}\leq V_{-m}, V_{-m}>K  \right) \\
&&\quad+\,\sum_{m=1}^{\infty}\PP \left( {\st}_{-m} > (x+m\alpha)(1-\rho ), T_{-m}\leq V_{-m}\leq K \right) \\
&&\leq\,\PP(V>K)\, \sum_{m=1}^{\infty}\PP \left( {\st}_{-m} > (x+m\alpha)(1-\rho )\right) \\
&&\quad+\,\PP( T_{-M}\leq K)\sum_{m>M}^{\infty}\PP \left( {\st}_{-m} > (x+m\alpha)(1-\rho ) \right)
+o(\PP(\sfw>x)) \\
&&=\,\left(\PP(V>K)+\PP( T_{-M}\leq K)\right)\,\frac{\rho}{1-\rho}\PP(\sfw>x),
\enarno
which may be neglected after first taking $M\to\infty$ and then $K\to\infty$.
We have used that the sum of a finite number of terms in the above
summations is of the order $o(\PP(\sfw>x))$, 
see Property~(\ref{1.3}) in Appendix~\ref{app:classes}.
This property, as well as other steps taken in the 
proof of Theorem~\ref{thm:srp}, will be used frequently in the following.
We have thus proved that, as $
x\rightarrow\infty$,
\[
\PP (W\ROS'>x)
\sim
\sum_{m=1}^{\infty}
\PP \left(
W\ROS'>x, {\st}_{-m} > (
x+m\alpha)
(1-\rho ) ,V_{-m}<T_{-m}\right)
+ o(\PP(\sfw>x)).
\]

\paragraph{Part I.}
We start with the scenario that the large customer is in service 
for the entire interval $(0,x)$.
This is the case when the 
workload $V_{-m}<T_{-m}$ and $B_{-m}>T_{-m}-V_{-m}+x$.
This immediately implies that the waiting time of customer 0 exceeds $x$.
By Theorem~\ref{thm:srp},
we have:
\bearno
&&\sum_{m=1}^{\infty}\PP \left(W\ROS'>x, {\st}_{-m} > (x+m\alpha)(1-\rho ) ,V_{-m}<T_{-m}, B_{-m}>T_{-m}-V_{-m}+x\right) \\
&&\,=\sum_{m=1}^{\infty}\PP \left({\st}_{-m} > (x+m\alpha)(1-\rho ) ,V_{-m}<T_{-m}, B_{-m}>T_{-m}-V_{-m}+x\right) \\
&&\,\sim\,\PP(\srp>x).
\enarno
This corresponds to the first term in Theorem~\ref{thm:main}.

\paragraph{Part II.}
We now investigate the event that $W\ROS'>x$ occurs while the large customer (customer~$-m$) is still in service at time~0, but not anymore at time $x$.
Thus, $V_{-m}<T_{-m}<B_{-m}+V_{-m}<T_{-m}+x$.
Let $W\ROS'(q)$ be the remaining waiting time of customer~0 after the first service completion after time~0 if $q$ is the number of competing customers at that instant.
In the sequel we shall simply write $W\ROS'(q)$ instead of $W\ROS'(\lfloor q\rfloor)$ when $q$ is not an integer.
If customer $-m$ is still in service at time 0, we may write
\[
    W\ROS' = V_{-m}+B_{-m}-T_{-m}+W\ROS'(A(T_{-m},T_{-m}+V_{-m}+B_{-m})) ,
\]  
where $A(s,t)$ denotes the number of arrivals between times $s$ and $t$.
We obviously have
\[
    W\ROS' \geq B_{-m}-T_{-m}+W\ROS'(A(T_{-m},T_{-m}+B_{-m})) ,
\]
\begin{equation}\label{eq:bound2}
    W\ROS' \leq B_{-m}+W\ROS'(A(T_{-m},T_{-m}+V_{-m}+B_{-m})) .
\end{equation}
The first bound gives
\bear
    &&\sum_{m=1}^\infty
    \PP\left( W\ROS'>x, B_{-m}>(m\alpha+x)(1-\rho), V_{-m}<T_{-m}<B_{-m}+V_{-m}\leq T_{-m}+x\right) \nonumber\\
    &&\geq\,
    \sum_{m=1}^\infty
    \PP\left( B_{-m}-T_{-m}+W\ROS'(A(T_{-m},T_{-m}+B_{-m}))>x, B_{-m}>(m\alpha+x)(1-\rho)\right., \nonumber\\
    &&\qquad\qquad \left.V_{-m}<T_{-m}<B_{-m}+V_{-m}\leq T_{-m}+x\right) \nonumber\\
    &&\geq\,(1-\delta)^2
    \sum_{m=M}^\infty
    \PP\big( B_{-m}-m\alpha(1+\varepsilon)+W\ROS'(\frac{1}{\alpha}(1-\varepsilon)B_{-m})>x, B_{-m}>(m\alpha+x)(1-\rho), \nonumber\\
    &&\qquad\qquad\qquad\qquad m\alpha(1+\varepsilon)<B_{-m}\leq m\alpha(1-\varepsilon)-K+x\big) ,\label{eq:lower1}
\enar
where, for fixed $\varepsilon>0$, $\delta>0$, we have chosen $K>0$ such that $\PP(V>K)<\delta$ and $M>0$ such that $M\alpha(1-\varepsilon)>K$ and, for all $m\geq M$ and $y\geq M\alpha(1-\rho)$,
\[
    \PP\left((1-\varepsilon)m\alpha<T_{-m}<(1+\varepsilon)m\alpha,
        A(T_{-m},T_{-m}+y)\geq \frac{1}{\alpha}y(1-\varepsilon)\right)\geq1-\delta,
\]
which is possible by the strong law of large numbers.

Note that the summation in~(\ref{eq:lower1}) is actually truncated at $m=(x-K)/(2\varepsilon)$.
For notation it is convenient to make the summation run from $m=0$ to $\infty$.
Adding the terms for $m<M$ in~(\ref{eq:lower1}) causes an 
error of the order $o(\PP(\sfw>x))$ and since
\bearno
    \sum_{m=0}^\infty \PP(m\alpha(1-\varepsilon)+x-K<B_{-m}<m\alpha+x)
    &\leq& \varepsilon O(\PP(\sfw>x)),\\
    \sum_{m=0}^\infty \PP(B_{-m}>(m\alpha+x)(1-\rho),m\alpha<B_{-m}<m\alpha(1+\epsilon)) 
    &\leq& \varepsilon O(\PP(\sfw>x)),
\enarno
we have,
\bearno
    &&\sum_{m=1}^\infty
    \PP\left( W\ROS'>x, B_{-m}>(m\alpha+x)(1-\rho), V_{-m}<T_{-m}<B_{-m}+V_{-m}\leq T_{-m}+x\right) \nonumber\\
       && +\frac{\varepsilon}{(1-\delta)^2} O(\PP(\sfw>x))\\
    &&\geq\,
    \sum_{m=0}^\infty
    \PP\big( B_{-m}-m\alpha(1+\varepsilon)+W\ROS'(\frac{1}{\alpha}(1-\varepsilon)B_{-m})>x, B_{-m}>(m\alpha+x)(1-\rho), \nonumber\\
    &&\qquad\qquad\qquad\qquad m\alpha<B_{-m}\leq m\alpha+x\Big) .
\enarno
Furthermore, by Lemma~\ref{lem:wq} we can find $x_0$ such that for all $x>x_0$:
\bearno
    &&    \int_{z=\max\{(m\alpha +x)(1-\rho),m\alpha \}}^{m\alpha  +x} \dd\PP(B\leq z)\,
    \PP\left( W\ROS'(\frac{1}{\alpha }(1-\varepsilon)z)>x-z+m\alpha (1+\varepsilon)\right) \\
    && \geq (1-\delta)\int_{z=\max\{(m\alpha +x)(1-\rho),m\alpha \}}^ {m\alpha  +x} \dd\PP(B\leq z)\,
    \left(\left(1-\frac{(1-\rho)(x-z+m\alpha (1+\varepsilon))}{\rho(1-\varepsilon)z}\right)^+\right)^{\frac{1}{1-\rho}} \\
    && \geq (1-\delta)(1-\gamma)\int_{z=\max\{(m\alpha +x)(1-\rho),m\alpha \}}^ {m\alpha  +x} \dd\PP(B\leq z)\,
    \left(\left(1-\frac{(1-\rho)(x-z+m\alpha )}{\rho z}\right)^+\right)^{\frac{1}{1-\rho}} ,
\enarno
where $\gamma>0$ depends on $\epsilon$.
In the last step we used that, as $\varepsilon\to0$,
\[
    \frac{(1-\rho)(x-z+m\alpha (1+\varepsilon))}{\rho(1-\varepsilon)z}
    \rightarrow
    \frac{(1-\rho)(x-z+m\alpha )}{\rho z},
\]
uniformly in $z$  within the area of integration.
(This can be seen, using that $z\ge(m\alpha +x)(1-\rho)$ and $z\ge m\alpha $.)
So that, with $\delta\rightarrow0$, $\varepsilon\to0$ and then $\gamma\rightarrow0$ we have:
\bear
    &&\sum_{m=1}^\infty
    \PP\left( W\ROS'>x, B_{-m}>(m\alpha +x)(1-\rho), V_{-m}<T_{-m}<B_{-m}+V_{-m}\leq T_{-m}+x\right) \nonumber\\
    &&\geq\,
    \sum_{m=0}^\infty
    \int_{z=\max\left\{(m\alpha +x)(1-\rho),m\alpha \right\}}^{m\alpha +x}
    \left(1-\frac{(1-\rho)(x+m\alpha -z)}{\rho z}\right)^{\frac{1}{1-\rho}}
    d \PP\left(B\leq z\right)
    \nonumber\\
    &&\qquad     +    o(\PP(\sfw>x)).\nonumber\\
    \nonumber
\enar
Next we derive an upper bound, 
\bear
    &&\sum_{m=1}^\infty
    \PP\left( W\ROS'>x, B_{-m}>(m\alpha +x)(1-\rho), V_{-m}<T_{-m}<B_{-m}+V_{-m}\leq T_{-m}+x\right) \nonumber\\
    &&=\,
    \sum_{m=1}^\infty 
    \PP\left( V_{-m}+B_{-m}-T_{-m}+W\ROS'(A(T_{-m},T_{-m}+V_{-m}+B_{-m}))>x, \right. \nonumber\\
    &&\qquad\qquad \left. B_{-m}>(m\alpha +x)(1-\rho), V_{-m}<T_{-m}<B_{-m}+V_{-m}\leq T_{-m}+x\right) \nonumber\\
    &&\leq\,
    \PP(V> K)\, O(\PP(\sfw>x)) \nonumber \\
    &&\quad+\sum_{m=1}^\infty 
    \PP\left( K+B_{-m}-T_{-m}+W\ROS'(A(T_{-m},T_{-m}+K+B_{-m}))>x, \right. \nonumber\\
    &&\qquad\qquad \left. B_{-m}>(m\alpha +x)(1-\rho), T_{-m}-K<B_{-m}\leq T_{-m}+x\right) 
    \nonumber\\
    &&\leq\,
    \left(\delta+\PP(V> K)\right)\, O(\PP(\sfw>x)) \nonumber \\
    &&\quad+\sum_{m=M}^\infty 
    \PP\Big( K+B_{-m}-m\alpha (1-\varepsilon)+W\ROS'(\frac{1+\varepsilon}{\alpha }(K+B_{-m}))>x, \nonumber\\
    &&\qquad\qquad B_{-m}>(m\alpha +x)(1-\rho), m\alpha (1-\varepsilon)-K<B_{-m}\leq m\alpha (1+\varepsilon)+x\Big) ,
    \nonumber
\enar
where, for fixed $\varepsilon>0$, we have chosen $M>0$ such that, for all $m\geq M$ and $y\geq M\alpha (1-\rho)$,
\[
    \PP\left((1-\varepsilon)m\alpha <T_{-m}<(1+\varepsilon)m\alpha ,
        A(T_{-m},T_{-m}+y)\leq \frac{1}{\alpha }y(1+\varepsilon)\right)\geq1-\delta.
\]
As in the lower bound we may let the summation run from $m=1$ to $\infty$ and replace the condition $m\alpha (1-\varepsilon)-K<B_{-m}\leq m\alpha (1+\varepsilon)+x$ with $m\alpha -K<B_{-m}\leq m\alpha +x-K$; the error we make is of the order $\epsilon\,O(\PP(\sfw>x))$.
Also, replacing $B>(x+m\alpha )(1-\rho)$ by $B>(x+m\alpha )(1-\rho)-K$ does not decrease the probability.
\bear
    &&\sum_{m=1}^\infty
    \PP\left( W\ROS'>x, B_{-m}>(m\alpha +x)(1-\rho), V_{-m}<T_{-m}<B_{-m}+V_{-m}\leq T_{-m}+x\right) \nonumber\\
    &&\leq\,
    \left(\epsilon+\delta+\PP(V> K)\right)\, O(\PP(\sfw>x)) \nonumber \\
    &&\quad+\sum_{m=0}^\infty 
    \PP\Big( K-m\alpha (1-\varepsilon)+B_{-m}+W\ROS'(\frac{1+\varepsilon}{\alpha }(K+B_{-m}))>x, \nonumber\\
    &&\qquad\qquad B_{-m}>(m\alpha +x)(1-\rho)-K, m\alpha -K<B_{-m}\leq m\alpha +x-K\Big)
    \nonumber\\
    &&=\,
    \left(\epsilon+\delta+\PP(V> K)\right)\, O(\PP(\sfw>x)) \nonumber \\
    &&\quad+\sum_{m=0}^\infty 
    \int_{z=\max\{(m\alpha +x)(1-\rho),m\alpha \}}^{m\alpha +x} \dd \PP(B\leq z-K)\,
    \PP\Big( W\ROS'(\frac{1+\varepsilon}{\alpha }z)>x+m\alpha (1-\varepsilon)-z\Big)
    \nonumber\\
    &&\leq\,
    \left(\epsilon+\delta+\PP(V> K)\right)\, O(\PP(\sfw>x)) \nonumber \\
    &&\quad+(1+\gamma)\sum_{m=0}^\infty 
    \int_{z=\max\{(m\alpha +x)(1-\rho),m\alpha \}}^{m\alpha +x} \dd \PP(B\leq z-K)\,
    \left(\left(1-\frac{(1-\rho)(x+m\alpha -z)}{\rho z}\right)^+\right)^{\frac{1}{1-\rho}}
    .
    \nonumber
\enar
In the last step we use Lemma~\ref{lem:wq} and the uniform convergence of
\[
    \frac{(1-\rho)(x+m\alpha (1-\varepsilon)-z)}{\rho (1+\varepsilon)z}
\]
as $\varepsilon\to0$ ($\gamma>0$ depends on $\varepsilon$).
\\
Using that
\[
    \sum_{m=0}^\infty \PP(\max\{(m\alpha +x)(1-\rho),m\alpha \}-K<B<\max\{(m\alpha +x)(1-\rho),m\alpha \})
    =
    o(\PP(\sfw>x)),
\]
and
\[
    \sum_{m=0}^\infty \PP(m\alpha +x-K<B<m\alpha +x)
    =
    o(\PP(\sfw>x)),
\]
we may replace $\dd \PP(B>z-K)$ by $\dd \PP(B>z)$.
Now let $K\to\infty$, $\epsilon\to0$, $\delta\to0$ and then $\gamma\to0$ to conclude that
\bear
    &&\sum_{m=1}^\infty
    \PP\left( W\ROS'>x, B_{-m}>(m\alpha +x)(1-\rho), V_{-m}<T_{-m}<B_{-m}+V_{-m}\leq T_{-m}+x\right) \nonumber\\
    &&\leq\,
    o(\PP(\sfw>x)) +\sum_{m=0}^\infty 
    \int_{z=\max\{(m\alpha +x)(1-\rho),m\alpha \}}^{m\alpha +x} \dd \PP(B\leq z)\,
    \left(1-\frac{(1-\rho)(x+m\alpha -z)}{\rho z}\right)^{\frac{1}{1-\rho}}
    .
    \nonumber
\enar
Together with the lower bound, this shows
\bear
    &&\sum_{m=1}^\infty
    \PP\left( W\ROS'>x, B_{-m}>(m\alpha +x)(1-\rho), V_{-m}<T_{-m}<B_{-m}+V_{-m}\leq T_{-m}+x\right) \nonumber\\
    &&=\,
    o(\PP(\sfw>x)) +\sum_{m=0}^\infty 
    \int_{z=\max\{(m\alpha +x)(1-\rho),m\alpha \}}^{m\alpha +x} \dd \PP(B\leq z)\,
    \left(1-\frac{(1-\rho)(x+m\alpha -z)}{\rho z}\right)^{\frac{1}{1-\rho}}
    .
    \nonumber
\enar
The second and third term in Theorem~\ref{thm:main} now readily follow, using that
\bearno
    &&\sum_{m=0}^\infty \int_{z=\max\{(m\alpha +x)(1-\rho),m\alpha \}}^{m\alpha +x} \dd \PP(B\leq z)\,
    \left(1-\frac{(1-\rho)(x+m\alpha -z)}{\rho z}\right)^{\frac{1}{1-\rho}}
    \\
    &&=\,\int_{v=0}^{\infty}\dd v \int_{z=\max\{(v\alpha +x)(1-\rho),v\alpha \}}^{v\alpha +x} \dd \PP(B\leq z)\,
    \left(1-\frac{(1-\rho)(x+v\alpha -z)}{\rho z}\right)^{\frac{1}{1-\rho}}\\
    &&\qquad+\,o(\PP(\sfw>x)).
\enarno

\paragraph{Part III.}
Finally, we deal with the last possible scenario in which service of the large customer ends before time~0.
Thus, $V_{-m}+B_{-m}<T_{-m}$.
Suppose that customer $-m+N$ is in service at time~0; $1\leq N\leq m-1$.
We can bound the waiting time of customer~0 from below by
\bearno
    W\ROS' &\ge & 
    W\ROS'(A(-T_{-m},0)-N)
    \,=\,W\ROS'(m-N),
\enarno
and from above by
\bearno
    W\ROS'&\leq& B_{-m+N}+W\ROS'(m-N+1).
\enarno
We start with the lower bound.
We shall denote the number of departures in the interval $[u,v)$ by $D(u,v)$.
Note that $N=D(-T_{-m}+V_{-m}+B_{-m},0)\leq D(-T_{-m}+B_{-m},0)$.
In the following we take $\varepsilon$, $\delta$, $M$ and $K$ such that $\PP(V>K)<\delta$ and for all $m\ge M$, $y\geq K$,
\[
    \PP\left((1-\epsilon)m\alpha <T_{-m}<(1+\varepsilon)m\alpha , 
    D(-y,0)>(1-\varepsilon)\frac{y}{\beta}\right)>1-\delta.
\]
We have,
\bearno
    &&\sum_{m=0}^\infty \PP\left(W\ROS'>x,B_{-m}>(m\alpha +x)(1-\rho),V_{-m}+B_{-m}<T_{-m}\right) 
    \\
    &&\geq \,(1-\delta)^2 \sum_{m=M}^\infty
    \PP\Big(W\ROS'(m-(1-\varepsilon)\frac{1}{\beta}\left((1+\varepsilon)m\alpha -B_{-m}\right))>x,\\
    &&\qquad\qquad \qquad\qquad \qquad\qquad B_{-m}>(m\alpha +x)(1-\rho),K+B_{-m}<(1-\varepsilon)m\alpha \Big)
    \\
    &&= \,\varepsilon O(\PP(\sfw>x))
    \,+\,(1-\delta)^2 \sum_{m=0}^\infty
    \PP\Big(W\ROS'(m-(1-\varepsilon)\frac{1}{\beta}\left((1+\varepsilon)m\alpha -B_{-m}\right))>x,\\
    &&\qquad\qquad \qquad\qquad \qquad\qquad \qquad\qquad \qquad\qquad B_{-m}>(m\alpha +x)(1-\rho),B_{-m}<m\alpha \Big)
    \\
    &&= \,\varepsilon O(\PP(\sfw>x))\\
    &&\qquad+\,(1-\delta)^2 \sum_{m=0}^\infty
    \int_{z=(m\alpha +x)(1-\rho)}^{m\alpha } \dd \PP(B\leq z)\,
    \PP\Big(W\ROS'(m-(1-\varepsilon)\frac{1}{\beta}\left((1+\varepsilon)m\alpha -z\right))>x\Big)
    \\
    &&\geq \,\varepsilon O(\PP(\sfw>x))\\
    &&\qquad+\,(1-\gamma)(1-\delta)^2 \sum_{m=0}^\infty
    \int_{z=(m\alpha +x)(1-\rho)}^{m\alpha } \dd \PP(B\leq z)\,
    \left(1-\frac{(1-\rho)x}{z-m\alpha (1-\rho)}    \right)^\frac{1}{1-\rho},
\enarno
where $\gamma>0$ depends on $\epsilon$.
In the last step we used Lemma~\ref{lem:wq} and the uniform convergence of
\[
    \frac{(1-\rho)x}{\beta(m-(1-\varepsilon)\frac{1}{\beta}\left((1+\varepsilon)m\alpha -z\right))}
\]
as $\varepsilon\to0$.
Similar to Part II it can be shown that
\bearno
    &&\sum_{m=0}^\infty
    \int_{z=(m\alpha +x)(1-\rho)}^{m\alpha } \dd \PP(B\leq z)\,
    \left(1-\frac{(1-\rho)x}{z-m\alpha (1-\rho)}    \right)^\frac{1}{1-\rho} \\
    &&=\,o(\PP(\sfw>x))+
    \int_{v=0}^\infty
    \int_{z=(v\alpha +x)(1-\rho)}^{v\alpha } \dd \PP(B\leq z)\,
    \left(1-\frac{(1-\rho)x}{z-v\alpha (1-\rho)}    \right)^\frac{1}{1-\rho}.
\enarno
Letting $\varepsilon\to0$, $\delta\to0$ and $\gamma\to0$ we thus have proved that
\bearno
    &&\sum_{m=0}^\infty \PP\left(W\ROS'>x,B_{-m}>(m\alpha +x)(1-\rho),V_{-m}+B_{-m}<T_{-m}\right) 
    \\
    &&\geq \,o(\PP(\sfw>x))
    \,+\,
    \int_{v=0}^\infty
    \int_{z=(v\alpha +x)(1-\rho)}^{v\alpha } \dd \PP(B\leq z)\,
    \left(1-\frac{(1-\rho)x}{z-v\alpha (1-\rho)}    \right)^\frac{1}{1-\rho}.
\enarno
It remains to show that the right-hand side is also an upper bound for the left-hand side.
Recall that $N=D(-T_{-m}+V_{-m}+B_{-m},0)$ and $W\ROS'\leq B_{-m+N}+W\ROS'(m-N+1)$.
Note that, if $\varepsilon>0$ and $\delta>0$ and $M$ such that $\PP(T_{-m}>(1+\varepsilon)m\alpha )<\delta$ for all $m\ge M$, then
\bearno
    &&\sum_{m=0}^\infty \PP\left(B_{-m}>(m\alpha +x)(1-\rho),V_{-m}+B_{-m}<T_{-m},V_{-m}+B_{-m}>(1-2\varepsilon)m\alpha \right) 
    \\
    &&\leq\, (\delta +\PP(V>K)) O(\PP(\sfw>x))\\
    &&\qquad+ \sum_{m=M}^\infty \PP\left(B_{-m}>(m\alpha +x)(1-\rho),(1-2\varepsilon)m\alpha -K<B_{-m}<(1+\varepsilon)m\alpha \right) 
    \\
    &&=\, (\varepsilon + \delta +\PP(V>K)) O(\PP(\sfw>x)).
\enarno
We shall use this in what follows.
In addition, let $M$ and $K$ be such that $\PP(V>K)<\delta$, and for all $m\ge M$ and $y\ge\varepsilon M\alpha $,
\[
    \PP(D(-y,0)\leq (1-\varepsilon)y/\beta)<\delta.
\]
Also, we take $L$ such that $\PP(B>L)<\delta$.
\bearno
    &&\sum_{m=0}^\infty \PP\left(W\ROS'>x,B_{-m}>(m\alpha +x)(1-\rho),V_{-m}+B_{-m}<T_{-m}\right) 
    \\
    &&\leq\, (\varepsilon+\delta+\PP(V>K)+\PP(B>L)) O(\PP(\sfw>x))
    \\
    &&+\sum_{m=M}^\infty 
    \PP\left(W\ROS'(m-N+1)>x-L,B_{-m}>(m\alpha +x-L)(1-\rho),V_{-m}+B_{-m}<(1-2\varepsilon)m\alpha \right)     \\
    &&\leq\, (\varepsilon+\delta) O(\PP(\sfw>x+L))
    \\
    &&\quad+\sum_{m=M}^\infty     
    \PP\Big(W\ROS'(m-(1-\varepsilon)\frac{1}{\beta}((1-\varepsilon)m\alpha -K-B_{-m})+1)>x,\\
    &&\qquad\qquad B_{-m}>(m\alpha +x)(1-\rho),B_{-m}<(1-2\varepsilon)m\alpha \Big)     \\
    &&=\, (\varepsilon+\delta) O(\PP(\sfw>x))
    \\
    &&\quad+\sum_{m=0}^\infty     
    \PP\Big(W\ROS'(m-(1-\varepsilon)\frac{1}{\beta}((1-\varepsilon)m\alpha -K-B_{-m})+1)>x,\\
    &&\qquad\qquad B_{-m}>(m\alpha +x)(1-\rho)-K-1,B_{-m}<m\alpha -K-1\Big)     \\
    &&= \,(\varepsilon+\delta) O(\PP(\sfw>x))
    \,+\,
    (1+\gamma)
    \int_{v=0}^\infty
    \int_{z=(v\alpha +x)(1-\rho)}^{v\alpha } \dd \PP(B\leq z)\,
    \left(1-\frac{(1-\rho)x}{z-v\alpha (1-\rho)}    \right)^\frac{1}{1-\rho}
    .
\enarno
As before, in the last step we use Lemma~\ref{lem:wq}, the uniform convergence of
\[
    \frac{(1-\rho)x}{\beta m-(1-\varepsilon)((1-\varepsilon)m\alpha -K-B_{-m})+\beta},
\]
as $\varepsilon\to0$ and the fact that replacing the summation with an integral and $\dd \PP(B\leq z-K-1)$ with $\dd P(B\leq z)$ introduces an error of the order $o(\PP(\sfw>x))$.
Letting $\varepsilon\to0$, $\delta\to0$ and $\gamma\to0$ yields the last term in Theorem~\ref{thm:main}.
\qed

\section{Random and deterministic arrivals}
\label{app:constarr}

The following lemma states that when interested in events 
involving a large service time, we may in fact ignore the 
randomness in the arrival process and replace it by a deterministic 
arrival process with the same mean arrival rate.
Thus, heuristically, we may concentrate on the $D/G/1$ queue 
instead of the $GI/G/1$ queue.
Although the lemma is not explicitly used in the paper, it
has been very useful in guiding us to the proof of several of its results.
We formulate it here because we expect that a reduction to a deterministic
arrival process will often be helpful in proving tail asymptotics,
see also Baccelli and Foss \cite{BF}.

\begin{lemma}
\label{lem:constarr}
If $B=B_0$ has a finite first moment  and if its integrated tail
distribution belongs to ${\cal L}$, then, 
for any constant
$c>0$, as $x\to\infty$,
\bearno
    \sum_{m=0}^\infty \PP\left(B>x+cT_{-m}\right)
    &\sim&
    \sum_{m=0}^\infty \PP\left(B>x+cT_{-m},B>x+cm\alpha \right) \\
    &\sim&
    \sum_{m=0}^\infty \PP\left(B>x+cm\alpha \right) \\
    &\sim&
    \frac{\rho}{c}\PP\left(B^{fw}>x\right).
\enarno
\end{lemma}

\begin{proof}
From Appendix~\ref{app:classes} (Property~(\ref{1.3}) in
Appendix~\ref{app:classes}), the integrated tail 
distribution of $B$ belongs to ${\cal L}$ if and only if $B^{fw} \in
{\cal L}$.

{\it Lower bound.}
For any $\varepsilon \in (0,1)$, choose $R>0$ such that
$$
\inf_{m\geq 0} 
\PP\left(T_{-m}\leq m\alpha (1+\varepsilon) +R \right)\geq 1-
\varepsilon .
$$ 
Then, as $x\to\infty$,
\begin{eqnarray*}
 \sum_{m=0}^\infty \PP\left(B>x+cT_{-m}\right) &\geq &
 \sum_{m=0}^\infty \PP\left(B>x+cm\alpha (1+\varepsilon) + c R,
T_{-m}\leq m\alpha (1+\varepsilon) +R\right)\\
&\geq &
(1-\varepsilon )
\sum_{m=0}^\infty \PP\left(B>x+cm\alpha (1+\varepsilon) + c R \right)\\
&\geq &
\frac{1-\varepsilon}{1+\varepsilon}
 \frac{\rho}{c}\PP\left(B^{fw}>x+cR\right)\\
&\sim &
\frac{1-\varepsilon}{1+\varepsilon}
 \frac{\rho}{c}\PP\left(B^{fw}>x \right) .
\end{eqnarray*}
Letting $\varepsilon \downarrow 0$, we get the right lower bound.

{\it Upper bound.}
Fix any $\varepsilon \in (0,1)$. Since $T_{-m}$ are partial
sums of non-negative i.i.d. r.v.'s with a finite positive mean $\alpha$,
$$
K\equiv  \sum_{m=0}^{\infty} \PP
\left(T_{-m}\leq m\alpha (1-\varepsilon )\right) < \infty .
$$
Then, as $x\to\infty$,
\begin{eqnarray*}
\sum_{m=0}^\infty \PP\left(B>x+cT_{-m}\right) &\leq &
\sum_{m=0}^\infty \PP\left(B>x+cm\alpha (1-\varepsilon ),
T_{-m}\geq m\alpha (1-\varepsilon )\right) \\
&+& 
\sum_{m=0}^\infty \PP\left(B>x,
T_{-m}\leq m\alpha (1-\varepsilon )\right) \\
&\leq &
 \sum_{m=0}^\infty \PP\left(B>x+cm\alpha (1-\varepsilon )\right)
+ K \PP (B>x) \\
&\leq &
\frac{1}{1-\varepsilon} 
\frac{\rho}{c}\PP\left(B^{fw}>x-c\alpha \right) + 
 K \PP (B>x) \\
&\sim & \frac{1}{1-\varepsilon} 
\frac{\rho}{c}\PP\left(B^{fw}>x\right),
\end{eqnarray*}
since $B^{fw}\in {\cal L}$ and
$\PP (B>x) = o(\PP (B^{fw}>x))$, see Appendix~\ref{app:classes} (Properties~(\ref{1.2}) and (\ref{1.3})).
Letting $\varepsilon \downarrow 0$, we get an upper bound which
coincides with the lower bound. 

\end{proof}

\end{document}